\documentclass[11pt]{article}
\usepackage[margin=1in]{geometry}
\usepackage{amsmath}
\usepackage{mathrsfs}
\usepackage{amsfonts}
\usepackage[mathscr]{euscript}
\usepackage{amssymb}
\usepackage{amsthm}
\usepackage{appendix}
\usepackage{tikz-cd}
\usepackage{mathtools}
\usepackage{theoremref}
\usepackage{pdfpages}
\usepackage{hyperref}
\usepackage{comment}
\usepackage{dsfont}
\usepackage{multicol}
\usepackage{graphicx}
\usepackage{float}
\usepackage{enumitem}
\usepackage{subcaption}
\usepackage{array}
\usepackage{mhsetup}
\usepackage[misc]{ifsym}

\definecolor{dhcol}{rgb}{0,0.5,0}

\definecolor{cecol}{rgb}{0,0,0.5}

\newcommand{\z}{\mathbf{z}}
\newcommand{\oz}{\overline{z}}
\newcommand{\obz}{\overline{\mathbf{z}}}
\newcommand{\w}{\mathbf{w}}
\newcommand{\zw}{\mathbf{z},\mathbf{w}}

\newcommand{\obx}{\overline{\mathbf{x}}}
\newcommand{\mres}{%
  \,\raisebox{-.127ex}{\reflectbox{\rotatebox[origin=br]{-90}{$\lnot$}}}\,%
}

\begin{document}
\newcommand{\rf}[1]{(\ref{#1})}
\newcommand{\mmbox}[1]{\fbox{\ensuremath{\displaystyle{ #1 }}}}	
\newcommand{\hs}[1]{\hspace{#1mm}}
\newcommand{\vs}[1]{\vspace{#1mm}}
\newcommand{\ri}{{\mathrm{i}}}
\newcommand{\re}{{\mathrm{e}}}
\newcommand{\rd}{\mathrm{d}}
\newcommand{\R}{\mathbb{R}}
\newcommand{\Q}{\mathbb{Q}}
\newcommand{\N}{\mathbb{N}}
\newcommand{\Z}{\mathbb{Z}}
\newcommand{\C}{\mathbb{C}}
\newcommand{\K}{{\mathbb{K}}}
\newcommand{\cA}{\mathcal{A}}
\newcommand{\cB}{\mathcal{B}}
\newcommand{\cC}{\mathcal{C}}
\newcommand{\cS}{\mathcal{S}}
\newcommand{\cD}{\mathcal{D}}
\newcommand{\cH}{\mathcal{H}}
\newcommand{\cI}{\mathcal{I}}
\newcommand{\cItilde}{\tilde{\mathcal{I}}}
\newcommand{\cIhat}{\hat{\mathcal{I}}}
\newcommand{\cIcheck}{\check{\mathcal{I}}}
\newcommand{\cIstar}{{\mathcal{I}^*}}
\newcommand{\cJ}{\mathcal{J}}
\newcommand{\cM}{\mathcal{M}}
\newcommand{\cP}{\mathcal{P}}
\newcommand{\cV}{{\mathcal V}}
\newcommand{\cW}{{\mathcal W}}
\newcommand{\scrD}{\mathscr{D}}
\newcommand{\scrS}{\mathscr{S}}
\newcommand{\scrJ}{\mathscr{J}}
\newcommand{\sD}{\mathsf{D}}
\newcommand{\sN}{\mathsf{N}}
\newcommand{\sS}{\mathsf{S}}
 \newcommand{\sT}{\mathsf{T}}
 \newcommand{\sH}{\mathsf{H}}
 \newcommand{\sI}{\mathsf{I}}
\newcommand{\done}[2]{\dfrac{d {#1}}{d {#2}}}
\newcommand{\donet}[2]{\frac{d {#1}}{d {#2}}}
\newcommand{\pdone}[2]{\dfrac{\partial {#1}}{\partial {#2}}}
\newcommand{\pdonet}[2]{\frac{\partial {#1}}{\partial {#2}}}
\newcommand{\pdonetext}[2]{\partial {#1}/\partial {#2}}
\newcommand{\pdtwo}[2]{\dfrac{\partial^2 {#1}}{\partial {#2}^2}}
\newcommand{\pdtwot}[2]{\frac{\partial^2 {#1}}{\partial {#2}^2}}
\newcommand{\pdtwomix}[3]{\dfrac{\partial^2 {#1}}{\partial {#2}\partial {#3}}}
\newcommand{\pdtwomixt}[3]{\frac{\partial^2 {#1}}{\partial {#2}\partial {#3}}}
\newcommand{\bnabla}{\boldsymbol{\nabla}}
\newcommand{\dive}{\boldsymbol{\nabla}\cdot}
\newcommand{\curl}{\boldsymbol{\nabla}\times}
\newcommand{\Phixy}{\Phi(\bx,\by)}
\newcommand{\PhiOxy}{\Phi_0(\bx,\by)}
\newcommand{\dxPhixy}{\pdone{\Phi}{n(\bx)}(\bx,\by)}
\newcommand{\dyPhixy}{\pdone{\Phi}{n(\by)}(\bx,\by)}
\newcommand{\dxPhiOxy}{\pdone{\Phi_0}{n(\bx)}(\bx,\by)}
\newcommand{\dyPhiOxy}{\pdone{\Phi_0}{n(\by)}(\bx,\by)}
\newcommand{\eps}{\varepsilon}
\newcommand{\real}[1]{{\rm Re}\left[#1\right]} 
\newcommand{\im}[1]{{\rm Im}\left[#1\right]}
\newcommand{\ol}[1]{\overline{#1}}
\newcommand{\ord}[1]{\mathcal{O}\left(#1\right)}
\newcommand{\oord}[1]{o\left(#1\right)}
\newcommand{\Ord}[1]{\Theta\left(#1\right)}
\newcommand{\hsnorm}[1]{||#1||_{H^{s}(\bs{R})}}
\newcommand{\hnorm}[1]{||#1||_{\tilde{H}^{-1/2}((0,1))}}
\newcommand{\norm}[2]{\left\|#1\right\|_{#2}}
\newcommand{\normt}[2]{\|#1\|_{#2}}
\newcommand{\on}[1]{\Vert{#1} \Vert_{1}}
\newcommand{\tn}[1]{\Vert{#1} \Vert_{2}}
\newcommand{\xt}{\mathbf{x},t}
\newcommand{\PhiF}{\Phi_{\rm freq}}
\newcommand{\cone}{{c_{j}^\pm}}
\newcommand{\ctwo}{{c_{2,j}^\pm}}
\newcommand{\cthree}{{c_{3,j}^\pm}}
\newtheorem{thm}{Theorem}[section]
\newtheorem{lem}[thm]{Lemma}
\newtheorem{defn}[thm]{Definition}
\newtheorem{prop}[thm]{Proposition}
\newtheorem{cor}[thm]{Corollary}
\newtheorem{rem}[thm]{Remark}
\newtheorem{conj}[thm]{Conjecture}
\newtheorem{ass}[thm]{Assumption}
\newtheorem{example}[thm]{Example}
\newcommand{\tH}{\widetilde{H}}
\newcommand{\Hze}{H_{\rm ze}} 	
\newcommand{\uze}{u_{\rm ze}}		
\newcommand{\dimH}{{\rm dim_H}}
\newcommand{\dimB}{{\rm dim_B}}
\newcommand{\IntClosOm}{\mathrm{int}(\overline{\Omega})}
\newcommand{\IntClosOmOne}{\mathrm{int}(\overline{\Omega_1})}
\newcommand{\IntClosOmTwo}{\mathrm{int}(\overline{\Omega_2})}
\newcommand{\Ccomp}{C^{\rm comp}}
\newcommand{\tCcomp}{\tilde{C}^{\rm comp}}
\newcommand{\uC}{\underline{C}}
\newcommand{\utC}{\underline{\tilde{C}}}
\newcommand{\oC}{\overline{C}}
\newcommand{\otC}{\overline{\tilde{C}}}
\newcommand{\capcomp}{{\rm cap}^{\rm comp}}
\newcommand{\Capcomp}{{\rm Cap}^{\rm comp}}
\newcommand{\tcapcomp}{\widetilde{{\rm cap}}^{\rm comp}}
\newcommand{\tCapcomp}{\widetilde{{\rm Cap}}^{\rm comp}}
\newcommand{\hcapcomp}{\widehat{{\rm cap}}^{\rm comp}}
\newcommand{\hCapcomp}{\widehat{{\rm Cap}}^{\rm comp}}
\newcommand{\tcap}{\widetilde{{\rm cap}}}
\newcommand{\tCap}{\widetilde{{\rm Cap}}}
\newcommand{\ccap}{{\rm cap}}
\newcommand{\ucap}{\underline{\rm cap}}
\newcommand{\uCap}{\underline{\rm Cap}}
\newcommand{\cCap}{{\rm Cap}}
\newcommand{\ocap}{\overline{\rm cap}}
\newcommand{\oCap}{\overline{\rm Cap}}
\DeclareRobustCommand
{\mathringbig}[1]{\accentset{\smash{\raisebox{-0.1ex}{$\scriptstyle\circ$}}}{#1}\rule{0pt}{2.3ex}}
\newcommand{\cirH}{\mathringbig{H}}
\newcommand{\cirHs}{\mathringbig{H}{}^s}
\newcommand{\cirHt}{\mathringbig{H}{}^t}
\newcommand{\cirHm}{\mathringbig{H}{}^m}
\newcommand{\cirHzero}{\mathringbig{H}{}^0}
\newcommand{\deO}{{\partial\Omega}}
\newcommand{\OO}{{(\Omega)}}
\newcommand{\Id}{{\mathrm{Id}}}
\newcommand{\gap}{\mathrm{Gap}}
\newcommand{\ggap}{\mathrm{gap}}
\newcommand{\isom}{{\xrightarrow{\sim}}}
\newcommand{\half}{{1/2}}
\newcommand{\mhalf}{{-1/2}}
\newcommand{\inter}{{\mathrm{int}}}

\newcommand{\Hsp}{H^{s,p}}
\newcommand{\Htq}{H^{t,q}}
\newcommand{\tHsp}{{{\widetilde H}^{s,p}}}
\newcommand{\SP}{\ensuremath{(s,p)}}
\newcommand{\Xsp}{X^{s,p}}

\newcommand{\dd}{{d}}\newcommand{\pp}{{p_*}}

\newcommand{\Rnn}{\R^{n_1+n_2}}
\newcommand{\Tr}{{\mathrm{Tr}}}

\title{Semi-discrete optimal transport methods\linebreak for the semi-geostrophic equations}
\date{}
\author{David P. Bourne\footnote{Maxwell Institute for Mathematical Sciences and Department of Mathematics, Heriot-Watt University, Edinburgh, United Kingdom.
\Letter\ d.bourne@hw.ac.uk, cpe4@hw.ac.uk, b.pelloni@hw.ac.uk.} , Charlie P. Egan\footnotemark[1]\ , Beatrice Pelloni\footnotemark[1] , Mark Wilkinson\thanks{Department of Mathematics, Nottingham Trent University, Nottingham, United Kingdom.\newline
\Letter\ mark.wilkinson@ntu.ac.uk.}}

\maketitle

\begin{abstract}
We give a new and constructive proof of the existence of global-in-time weak solutions of the 3-dimensional incompressible semi-geostrophic equations (SG) in geostrophic coordinates, for arbitrary initial measures with compact support. This new proof, based on semi-discrete optimal transport techniques, works by characterising discrete solutions of SG in geostrophic coordinates in terms of trajectories satisfying an ordinary differential equation. It is advantageous in its simplicity and its explicit relation to Eulerian coordinates through the use of Laguerre tessellations. Using our method, we obtain improved time-regularity for a large class of discrete initial measures, and we compute explicitly two discrete solutions. The method naturally gives rise to an efficient numerical method, which we illustrate by presenting simulations of a 2-dimensional semi-geostrophic flow in geostrophic coordinates generated using a numerical solver for the semi-discrete optimal transport problem coupled with an ordinary differential equation solver.
\end{abstract}

\section{Introduction}
The incompressible semi-geostrophic equations (SG) model the large-scale dynamics of rotational atmospheric flows. They can be viewed as a low Rossby number limit of the primitive equations, and are used by meteorologists to diagnose irregularities in simulated Navier-Stokes flows of the atmosphere on length scales of the order of tens of kilometres (see Cullen \cite{jp2006mathematical} and Visram, Cotter and Cullen \cite{visram2014framework}). First proposed by Eliassen \cite{eliassen1949quasi} in 1949, and subsequently developed by Hoskins \cite{hoskins1975geostrophic} in 1975, the semi-geostrophic equations have attracted significant attention from the mathematical community over the past twenty years owing partly to their connection with optimal transport theory 
(see \cite{ambrosio2014global,ambrosio2012existence, benamou1998weak, cheng2018classical, cullen2006lagrangian, cullen2001variational, cullen2007semigeostrophic, faria2009weak, feldman2013lagrangian, feldman2015semi, feldman2017semi, lisai2020smooth, loeper2006fully, lopes2002existence, oneil2020rigorous}).

In this paper we consider SG in \emph{geostrophic coordinates}, associated to flows on an arbitrary convex bounded (physical) domain $\Omega\subset\mathbb{R}^{3}$, which we interpret as the active transport equation
\begin{align}\label{eqn:SGgeoActTrans}
\partial_t  \alpha_t + \mathcal{W}[\alpha_t] \cdot \nabla \alpha_t =0
\end{align}
for the time-dependent measure-valued map $\alpha$, which is known as the \emph{potential vorticity}. The connection between SG and optimal transport is contained in the nonlocal divergence-free velocity field $\mathcal{W}[\alpha]$, which is often called the \emph{geostrophic velocity}. Let $\mathcal{B}[\alpha_t]$ denote the unique mean-zero convex function whose gradient is the optimal transport map between the Lebesgue measure on $\Omega$ and the Borel measure $\alpha_t$ with respect to the quadratic cost, and let $\mathcal{B}[\alpha_t]^*$ denote its Legendre-Fenchel transform on $\Omega$. At each time $t$, $\mathcal{W}[\alpha]$ is given by
\begin{align*}
\mathcal{W}[\alpha_t]:=J(\mathrm{id}_{\R^3}-\nabla \mathcal{B}[\alpha_t]^*),
\end{align*}
where $\mathrm{id}_{\R^3}$ denotes the identity on $\R^3$ and
\begin{align}\label{eqn:J}
J:=\left(
\begin{array}{c c c}
0 & -1 & 0 \\
1 & 0 & 0 \\
0 & 0 & 0 \\
\end{array}
\right).
\end{align}
Guided by the work of Cullen and Purser \cite{cullen1984extended}, this connection was first established rigorously by Benamou and Brenier. In \cite{benamou1998weak}, those authors proved the existence of global-in-time weak solutions of \eqref{eqn:SGgeoActTrans} for initial measures which are absolutely continuous with respect to the Lebesgue measure and have compactly-supported $L^{p}$ density for $p>3$. This result was extended by Lopes Filho and Nussenzveig Lopes in \cite{lopes2002existence} to the case where $p\geq 1$, and by Loeper \cite{loeper2006fully} and later Feldman and Tudorascu \cite{feldman2015semi} to the case where the initial measure need only have compact support in $\R^3$. In \cite[Proposition 4.14]{feldman2013lagrangian}, Feldman and Tudorascu use Ambrosio and Gangbo's abstract techniques for Hamiltonian ODEs in Wasserstein space \cite{ambrosio2008hamiltonian} to prove that when the initial measure is an arbitrary convex combination of Dirac masses there exists a global-in-time solution that maintains the discrete structure of the initial data. The two known results regarding the uniqueness of solutions of \eqref{eqn:SGgeoActTrans} are the local-in-time uniqueness of H\" older continuous periodic solutions proved in \cite{loeper2006fully}, and weak-strong uniqueness under uniform convexity proved in \cite{feldman2017semi}. Otherwise, the problem of uniqueness of solutions remains open.

The main contribution of this paper is an alternative proof the existence of global-in-time weak solutions of \eqref{eqn:SGgeoActTrans} for arbitrary compactly-supported initial measures, which uses recently developed techniques from semi-discrete optimal transport to treat the case where the initial measure is discrete (see Section \ref{sect:mainResults} and, in particular, Theorems \ref{thm:1} and \ref{thm:2}).  For a wide class of discrete initial measures our result recovers \cite[Proposition 4.14]{feldman2013lagrangian} with improved time regularity (twice continuously differentiable rather than Lipschitz) and uniqueness. More significantly, our application of semi-discrete optimal transport to SG illuminates an explicit and intuitive connection between geostrophic coordinates and corresponding flows in the physical domain $\Omega$. It also gives a constructive way of determining solutions explicitly, and it forms the basis of an effective numerical scheme, as we illustrate in Section \ref{sect:num}.

\subsection{SG in geostrophic coordinates and semi-discrete optimal transport}\label{sect:introDescription}
In this section, we describe our approach to studying \eqref{eqn:SGgeoActTrans} using \emph{semi-discrete} optimal transport, which is the special case of optimal transport in which the source measure is absolutely continuous with respect to the Lebesgue measure and the target measure is discrete. 

In recent years, semi-discrete optimal transport theory has seen significant expansion in its theoretical foundations (see
\cite{berman2020convergence,de2019differentiation,DieciWalsh2019,HartmannSchuhmacher2020,kitagawa2016convergence, LevySemiDiscrete2015,li2020quantitative,merigot2011multiscale,merigot2020quantitative,merigot2021optimal,meyron2019initialization}). It has also been applied to many diverse problems in the sciences, both within fluid dynamics \cite{gallouet2018lagrangian, levy2018notions} and elsewhere 
such as materials science \cite{BourneKokRoperSpanjer2020,BournePeletierRoper2014,Kuhn2020}, economics \cite[Chapter 5]{Galichon2016}, crowd dynamics \cite{LeclercMerigotSantambrogio2020} and image interpolation \cite{LevySemiDiscrete2015}.
Inspired by ideas in the original work of Cullen and Purser \cite{cullen1984extended} on piecewise constant solutions of semi-geostrophic slice models, we use semi-discrete optimal transport to analyse \eqref{eqn:SGgeoActTrans} in the special case where the initial potential vorticity $\alpha_0=\overline{\alpha}$ is a discrete measure, i.e.,
\begin{align*}
\overline{\alpha}=\sum_{i=1}^{N}m_{i}\delta_{\overline{z}_{i}}
\end{align*}
for some $m_{i}>0$ and $\overline{z}_{i}\in\mathbb{R}^{3}$. We show (Theorem \ref{thm:1}) that for \emph{well-prepared} discrete initial data (see Definition \ref{defn:wellPrep}) there exists a corresponding \emph{discrete solution} $t\mapsto \alpha_{t}$ of \eqref{eqn:SGgeoActTrans} of the form
\begin{align}\label{eqn:discSoln}
\alpha_{t}=\sum_{i=1}^{N}m_{i}\delta_{z_{i}(t)},
\end{align}
where the trajectories $z_i$ are twice continuously differentiable. Denoting by $\mathcal{L}^3$ the Lebesgue measure on $\R^3$ and by $\mathcal{L}^3\mres\Omega$ its restriction to $\Omega$, the optimal transport map between $\mathcal{L}^3\mres\Omega$ and $\alpha_t$ given by \eqref{eqn:discSoln} is a piecewise constant function
\begin{align*}
T=\sum_{i=1}^Nz_i\mathds{1}_{C_i},
\end{align*}
where $\{C_i\}_{i=1}^N$ is a tessellation of $\Omega$ by convex sets, known as the \emph{optimal Laguerre tesselation} (see Definitions \ref{def:LaguerreTess} and \ref{def:OWM}) generated by the \emph{seed vector} $\z:=(z_1,...,z_N)$ subject to the mass constraint
\begin{align*}
\mathcal{L}^3(C_i)=m_i\quad \forall\, i\in\{1,...,N\}.
\end{align*}
Letting $x_i(\z)$ denote the centroid of the \emph{Laguerre cell} $C_i(\z)$, one can show (see Lemma \ref{lem:odeChar}) that the time-dependent measure-valued map defined by \eqref{eqn:discSoln} is a weak solution of \eqref{eqn:SGgeoActTrans} if and only if the trajectories $z_1,...,z_N$ satisfy the ODE initial value problem (IVP)
\begin{align}\label{eqn:ODEIVPintro}
\begin{dcases}
\frac{dz_{i}}{dt}=J(z_{i}-x_{i}(\mathbf{z})),\\
z_i(0)=\overline{z}_i,
\end{dcases}
\end{align}
for $i\in\{1,...,N\}$. At each time $t$, the seeds $z_i(t)$ in geostrophic space generate a Laguerre tessellation of the physical domain $\Omega$ (see Figure \ref{fig:illustration} for a 2D illustration).

\begin{figure}
\begin{center}
\includegraphics[scale=0.4,trim={3cm 1cm 1cm 1cm},clip]{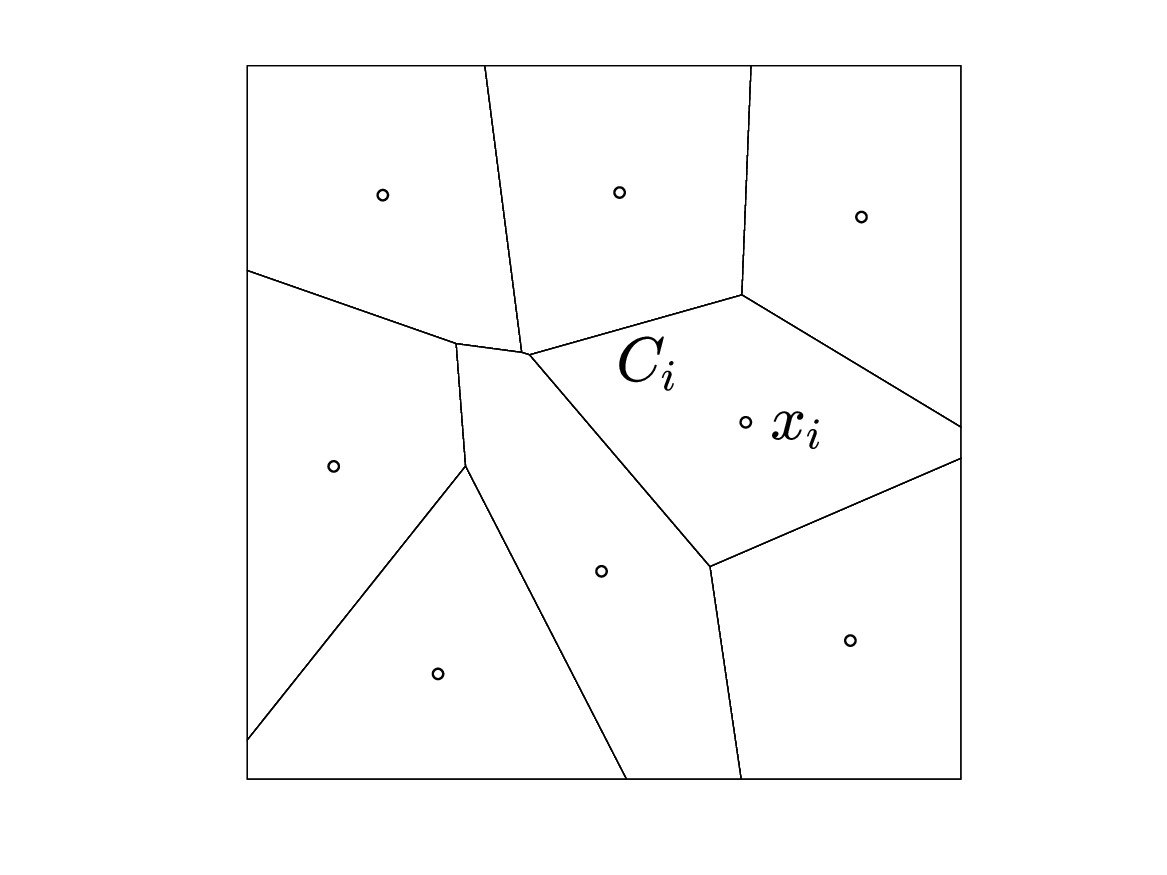}
\includegraphics[scale=0.4,trim={3cm 1cm 1cm 1cm},clip]{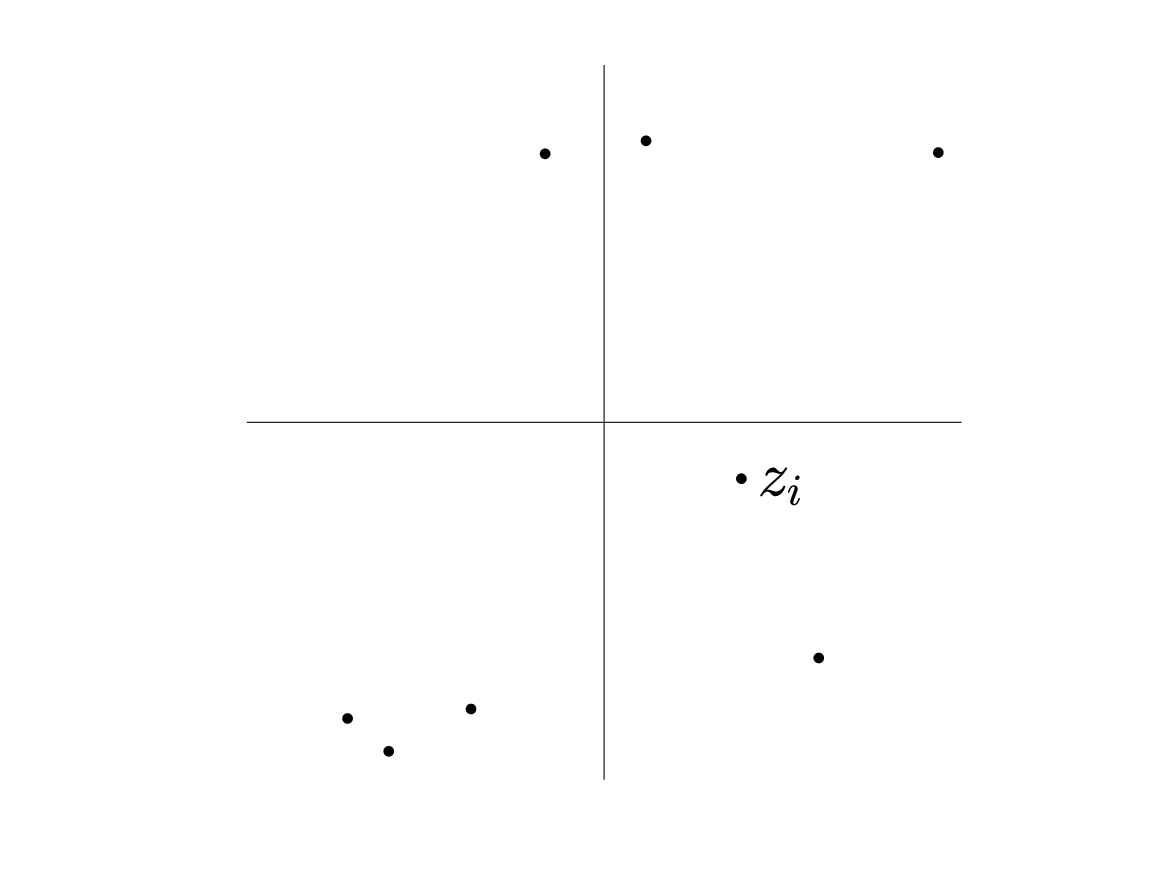}
\caption{\label{fig:illustration}Typical snapshot at a time $t$ of a 2D discrete solution $\alpha$ of \eqref{eqn:SGgeoActTrans} and the corresponding tessellation of the physical domain $\Omega$, which is taken here to be a square subset of $\R^2$. The right-hand plot shows a configuration of $N=8$ seeds $\z(t)=(z_1(t),...,z_8(t))$ in 2D geostrophic space, on which the measure $\alpha_t$ is supported. The left-hand plot shows the (approximate) Laguerre tessellation of $\Omega$ generated by $\z(t)$ subject to the constraint that all cells have the same mass. Black lines represent Laguerre cell boundaries and the circle within each cell $C_i$ represents its centroid $x_i(\z(t))$.}
\end{center}
\end{figure}

We obtain a solution of \eqref{eqn:SGgeoActTrans} for an arbitrary compactly-supported initial measure $\overline{\alpha}$ by generating a sequence $(\overline{\alpha}^N)_{N\in\N}$ of well-prepared discrete measures converging to $\overline{\alpha}$ in the Wasserstein 2-distance, evolving each of these discrete measures according to the corresponding ODE-IVP \eqref{eqn:ODEIVPintro}, and using compactness in the space of continuous measure-valued maps to pass to the limit as $N\to\infty$. As such, our construction method can be thought of as a \emph{meshless} or \emph{particle} method. By comparison with the proof given in \cite{benamou1998weak}, and later generalised in \cite{loeper2006fully} and \cite{lopes2002existence}, the discretisation occurs in the spatial domain rather than in the time domain. Note that in \cite{cullen2007semigeostrophic} Cullen, Gangbo and Pisante analyse a \emph{variant} of SG using a spatial discretisation different from the one considered in this paper. Analytically, the essential benefit of the discretisation of the initial measure $\overline{\alpha}$ is that the study of the active transport equation \eqref{eqn:SGgeoActTrans}, whose velocity field is only in general of class $BV_{\mathrm{loc}}$, is replaced by the study of the ODE-IVP \eqref{eqn:ODEIVPintro}, whose right hand side is continuously differentiable $\mathcal{L}^{3N}$-almost everywhere. Mollifications of the vector-field and related quantities used in \cite{benamou1998weak}, \cite{loeper2006fully} and \cite{lopes2002existence}, as well as the abstract techniques for Hamiltonian ODEs in Wasserstein space used in \cite{feldman2013lagrangian}, are therefore avoided, resulting in a more direct solution procedure.

\subsection{Background on the semi-geostrophic equations}

In this section, we briefly describe how equation \eqref{eqn:SGgeoActTrans} is derived. In their traditional Eulerian formulation in a fixed spatial domain $\Omega\subset\mathbb{R}^{3}$, the (non-dimensionalised) semi-geostrophic equations are given by the coupled system
\begin{align}\label{eqn:sgorig}
\left\{
\begin{array}{l}
\displaystyle \partial_{t}u_{g}+(u\cdot\nabla)u_{g}=-Ju_{a}, \vspace{2mm} \\
\displaystyle \partial_{t}\theta+(u\cdot\nabla)\theta=0, \vspace{2mm}\\
\nabla\cdot u=0,
\end{array}
\right.
\end{align}
where $u_{g}:=(u_{g, 1}, u_{g, 2}, 0)^{\mathrm{T}}$ is the \emph{geostrophic velocity field}, $u$ is the Eulerian velocity field of the fluid, $u_{a}:=u-u_{g}$ is the \emph{ageostrophic velocity field}, $\theta$ is the \emph{potential temperature}, and the matrix  $J$, defined by \eqref{eqn:J}, encodes planetary rotation. Importantly, the hydrodynamic and thermodynamic fields $u_{g}$ and $\theta$ are linked through the fluid pressure $p$ by the identity
\begin{align}\label{eqn:pressUg}
\nabla p=\left(
\begin{array}{c}
u_{g, 2} \\
-u_{g, 1} \\
\theta
\end{array}
\right).
\end{align}
When posed on a suitably smooth bounded domain $\Omega\subset\R^3$, the system \eqref{eqn:sgorig} is typically supplemented with a no-slip boundary condition $u\cdot n=0$ on $\partial\Omega$, where $n$ is the outward unit normal field on the boundary, which is sufficient to ensure that fluid points remain in the domain $\Omega$ for all times. Define the geopotential $P$ pointwise by 
\begin{align*}
P(x, t):=p(x, t)+\frac{1}{2}(x_{1}^{2}+x_{2}^{2})
\end{align*}
for $x\in\Omega$ and $t\geq 0$. The system \eqref{eqn:sgorig} can then be expressed as
\begin{align}\label{eqn:sgP}
\frac{\partial}{\partial t}\nabla P+(U[\nabla P]\cdot\nabla)\nabla P=J(\nabla P-\mathrm{id}_{\Omega}),
\end{align}
where $\mathrm{id}_{\Omega}$ denotes the identity map on $\Omega$, and $U:\nabla P\mapsto u$ is the formal solution operator associated to the (time-independent) div-curl boundary-value problem
\begin{align}\label{eqn:bvp}
\left\{
\begin{array}{l}
\nabla\wedge (D^{2}Pu)=\nabla\wedge J(\nabla P-\mathrm{id}_{\Omega}) \quad \text{in}\hspace{2mm}\Omega, \vspace{2mm} \\
\nabla\cdot u=0 \quad \text{in}\hspace{2mm}\Omega, \vspace{2mm}\\
u\cdot n=0 \quad \text{on}\hspace{2mm}\partial\Omega.
\end{array}
\right.
\end{align}
By way of this simple change of dependent variable, SG can be viewed as an inhomogeneous active transport equation \eqref{eqn:sgP} whose unknown $\nabla P$ is a time-dependent conservative vector field on $\Omega$. The Eulerian velocity field is then formally defined through the action of the solution operator $U$. This change of dependent variable also highlights a substantial mathematical difficulty one faces when constructing solutions of \eqref{eqn:sgP}: for the boundary-value problem \eqref{eqn:bvp} to be of elliptic type at each time $t$, $P(\cdot,t)$ must be strictly convex. 

The state-of-the-art regarding the existence of solutions of SG in Eulerian coordinates is due to Ambrosio, Colombo, De Philippis and Figalli \cite{ambrosio2014global}. Using the $W^{2, 1}_{\mathrm{loc}}$-regularity of Alexandrov solutions of a class of Dirichlet boundary-value problems for the Monge-Amp\`{e}re equation established in \cite{de2013w}, the authors proved the existence of global-in-time distributional solutions of SG in Eulerian coordinates posed on smooth convex domains $\Omega\subset\mathbb{R}^{3}$ for a class of initial geopotentials $P_0$ satisfying
\begin{align*}
\text{supp}\left(\nabla P_0\#\mathcal{L}^{3}\mres\Omega\right)=\R^3.
\end{align*}
However, for such solutions, the support of the pushforward measure $\nabla P(\cdot, t)\#\mathcal{L}^{3}\mres\Omega$ is the whole space $\R^3$ at each time $t$. The relation \eqref{eqn:pressUg} then implies that the temperature field $\theta$ satisfies $\theta(\cdot, t)\notin L^{\infty}(\Omega)$. Interpreted physically, this means that the atmospheric fluid is arbitrarily hot on sets of positive measure at all times. At the time of writing, the existence of either local-in-time or global-in-time distributional solutions of \eqref{eqn:sgP} for \emph{physical} initial data $P_0$ satisfying $\nabla P_{0}(\Omega)\subset\subset\mathbb{R}^{3}$ remains open.

Since the pioneering work of Hoskins \cite{hoskins1975geostrophic}, Cullen and Purser \cite{cullen1984extended}, and Benamou and Brenier \cite{benamou1998weak} on the semi-geostrophic equations, it has become customary to regard $\nabla P(\cdot, t)$ formally as a diffeomorphism between $\Omega$ and its image $\nabla P(\Omega, t)$ for each time $t$. The system \eqref{eqn:sgorig} is then transformed to the time-dependent coordinate system determined by $\nabla P$, known as \emph{geostrophic coordinates}. It is a remarkable property of SG that, as shown in \cite{benamou1998weak}, this formal change of coordinates yields a closed equation which is free of the field $u$. Indeed, under the assumption that $\nabla P$ is a smooth solution of \eqref{eqn:sgP} and $P(\cdot, t)$ is strictly convex at each time $t$, it can be shown that the time-dependent pushforward measure 
\begin{align*}
\alpha_{t}:=\nabla P(\cdot, t)\#\mathcal{L}^{3}\mres \Omega
\end{align*}
is a distributional solution of the active transport equation \eqref{eqn:SGgeoActTrans}.

\subsection{Outline of the paper}
We begin in Section \ref{sect:SDOT} with a brief introduction to semi-discrete optimal transport theory. Section \ref{sect:mainResults} contains the statement of the following existence results whose novel proofs are the main contribution of this paper: 
\begin{enumerate}
\item  \emph{discrete geostrophic solutions} with \emph{well-prepared} discrete initial data exist, are unique, and are defined by trajectories that are twice continuously differentiable in time (see Definition \ref{defn:wellPrep} and Theorem \ref{thm:1});
\item Lipschitz-in-time solutions of SG in geostrophic coordinates with arbitrary compactly-supported initial measure can be constructed as the uniform limit of a sequence of \emph{discrete geostrophic solutions} that are twice continuously differentiable in time (Theorem \ref{thm:2}).
\end{enumerate}
\noindent These results are proved in Sections \ref{sect:discExist} and \ref{sect:generalExistence} respectively. Section \ref{sect:explicit} contains the explicit calculation of two exact solutions of SG in geostrophic coordinates, as well as a brief discussion on equilibrium solutions. Finally, in Section \ref{sect:num}, we illustrate the theory developed in the paper by simulating a 2D semi-geostrophic flow in geostrophic coordinates, and we plot the corresponding Laguerre tessellations of the physical domain $\Omega$.
\subsection{Notation}

Let $d\in\N$. We denote by $\R^d_{>}$ the subset of $\R^d$ consisting of all vectors whose components are positive. For $i\in\{1,...,d\}$, the $i^{\text{th}}$ canonical basis vector in $\R^d$ is denoted by $e_i$. Let $A\subseteq \R^d$ be a Borel set. We denote the identity map on $A$ by $\mathrm{id}_A$, and the characteristic function of $A$ by $\mathds{1}_A$. We denote the interior of $A$ by $\mathrm{Int}(A)$ and the boundary of $A$ by $\partial A$.

\paragraph{Measures.}We denote by $\mathcal{L}^d$ the Lebesgue measure on $\R^d$ and by $\mathcal{L}^d\mres A$ its restriction to $A$. The set of Borel probability measures on $A$ is denoted by $\mathcal{P}(A)$. Given a Borel map $T:A\to\R^d$ and a measure $\mu\in\mathcal{P}(A)$, the pushforward of $\mu$ by $T$ is denoted by $T_\#\mu$ and is defined by $T_\#\mu(B)=\mu(T^{-1}(B))$ for all Borel sets $B\subseteq \R^d$. The set of Borel probability measures on $\R^d$ with compact support is denoted by $\mathcal{P}_c(\R^d)$. For any $p\in [1,+\infty)$, $\mathcal{P}_p(\R^d)$ denotes the set of all Borel probability measures on $\R^d$ with finite moments of order $p$, equipped with the Wasserstein $p$-distance $W_p$. This is defined for $\mu,\,\nu\in\mathcal{P}_p(\R^d)$ by
\begin{align*}
W_p(\mu,\nu):=\inf\left\{\int_{\R^d\times\R^d} |x-y|^p\,\rd\gamma(x,y)\,:\, \gamma\in\mathcal{P}(\R^d\times\R^d),\,{\pi_x}_\#\gamma=\mu\text,\,{\pi_y}_\#\gamma=\nu\right\}^{\frac{1}{p}},
\end{align*}
where $\pi_x$ and $\pi_y$ denote the projections onto the first and second variables, respectively. Throughout this paper spaces of probability measures are understood to be equipped with the Wasserstein 2-distance unless otherwise stated. 

\paragraph{Convex functions.} Given a convex function $f:A\to\R$, the subdifferential of $f$ is the set-valued function, mapping from $A$ into the set of subsets of $\R^d$, defined by
\begin{align*}
\partial f (x)=\left\{y\in\R^d \ \Big\vert\  y\cdot(z-x)\leq f(z)-f(x)\quad \forall\,z\in A \right\}.
\end{align*}
The Legendre-Fenchel transform of $f$ is the function $f^*:\R^d\to\R$ defined by
\begin{align*}
f^*(y)=\underset{x\in A}\sup\left\{x\cdot y-f(x)\right\}.
\end{align*}

\paragraph{Test functions.} We denote by $\mathscr{D}(\R^d)$ the space of test functions $C^{\infty}_c(\R^d)$ equipped with the standard semi-norm topology (see, for example, \cite{friedlander1998introduction}).

\paragraph{Physical domain.} Throughout this paper (with the exception of Section \ref{sect:num}) $\Omega$ is taken to be an arbitrary convex open bounded subset of $\R^3$, and, without loss of generality, we use the normalisation convention that $\mathcal{L}^3(\Omega)=1$ so that all measures under consideration are probability measures.

\section{Semi-discrete optimal transport}\label{sect:SDOT}
In this section we review some basic aspects of semi-discrete optimal transport theory. For further information on semi-discrete optimal transport see \cite[Section 4]{merigot2021optimal} and, for more on optimal transport theory in greater generality, see \cite{santambrogio2015optimal, villani2003topics, villani2008optimal}.

Given a \emph{target measure} $\nu\in\mathcal{P}_2(\R^3)$, a Borel map $T:\Omega\to\R^3$ is said to be an \emph{optimal transport map} between $\mathcal{L}^3\mres\Omega$ and $\nu$ with respect to the quadratic cost $c:\R^3\times\R^3\to\R$ given by
\begin{align*}
c(x,y)=|x-y|^2
\end{align*}
if it minimises the \emph{transport cost}
\begin{align*}
\int_{\Omega}|x-T(x)|^2\,\rd x
\end{align*}
subject to the constraint that
\begin{align*}
T\#\left(\mathcal{L}^3\mres\Omega\right)=\nu.
\end{align*}
The problem of finding an optimal transport map given source and target measures is known as the Monge problem. For any $\nu\in\mathcal{P}_2(\R^3)$ such a map $T$ exists, is unique, and can be expressed as the gradient of a convex function $\Phi$ belonging to the Sobolev space $H^1(\Omega)$ (see for example \cite[Theorem 1.22]{santambrogio2015optimal} or \cite[Theorem 2.12]{villani2003topics}).
\begin{defn}\label{def:BrenOper}
{\normalfont We define the operator $\mathcal{B}:\mathcal{P}_2(\R^3)\to H^1(\Omega)$ to be that which sends any given $\nu\in\mathcal{P}_2(\R^3)$ to the unique mean-zero convex function in $H^1(\Omega)$ whose gradient is the unique optimal transport map between $\mathcal{L}^3\mres\Omega$ and $\nu$ with respect to the quadratic cost.}
\end{defn}
\noindent The stability of optimal transport, including the continuity of $\mathcal{B}$, has been studied in, for example, \cite{berman2020convergence,li2020quantitative,merigot2020quantitative} and \cite[Theorem 5.20 and Corollary 5.23]{villani2008optimal}. In \cite{merigot2020quantitative} it is shown that for any bounded set $A\subset \R^3$, the restriction of $\mathcal{B}$ to $\mathcal{P}(A)$ is H\"{o}lder continuous.
\begin{thm}[{\normalfont c.f. \cite[Theorem 3.1]{merigot2020quantitative}}]\label{thm:quantStab}
Let $A\subset \R^3$ be bounded. There exists a constant $C>0$, which depends only on $A$ and $\Omega$, such that for all $\mu,\ \nu\in\mathcal{P}(A)$,
\begin{align*}
\|\nabla\mathcal{B}[\mu]-\nabla\mathcal{B}[\nu]\|_{L^2(\Omega;\R^3)}
\leqslant
CW_2(\mu,\nu)^{\frac{2}{15}}.
\end{align*}
\end{thm}

Semi-discrete optimal transport (see, for example, \cite{kitagawa2016convergence}, \cite[Section 4]{merigot2021optimal}, \cite[Section 6.4.2]{santambrogio2015optimal}) refers to the special case where, for some $N\in\N$, the target measure $\nu$ belongs to the class
\begin{align}\label{eqn:Q}
\mathcal{Q}^N\left(\R^3\right):=\left\{
\nu=\sum_{i=1}^Nm_i\delta_{z_i}\ \Big\vert\ \z\in D,\ \mathbf{m}=(m_1,...,m_N)\in\R^N_{>}\ \text{and}\ \sum_{i=1}^Nm_i=1\ \right\},
\end{align}
where
\begin{align}\label{eqn:D}
D:=\left\{\z=(z_1,...,z_N)\in\R^{3N}\ \big\vert\ z_i\neq z_j\text{ whenever } i\neq j\right\}.
\end{align}

\begin{defn}
{\normalfont We call a vector $\mathbf{m}=(m_1,...,m_N)\in \R^N_>$ such that $\sum_{i=1}^Nm_i=1$  a \emph{mass vector}, and a vector $\z\in D$ a \emph{seed vector}. A measure $\nu\in \mathcal{Q}^N\left(\R^3\right)$ given by
\begin{align*}
\nu=\sum_{i=1}^Nm_i\delta_{z_i}
\end{align*}
is said to have mass vector $\mathbf{m}=(m_1,...,m_N)$ and seed vector $\z=(z_1,...,z_N)$.}
\end{defn}

\noindent Let $\nu\in \mathcal{Q}^N\left(\R^3\right)$ have mass vector $\mathbf{m}=(m_1,...,m_N)$ and seed vector $\z=(z_1,...,z_N)$. A map $T:\Omega\to\R^3$ satisfies the pushforward constraint
\begin{align*}
T\#\left(\mathcal{L}^3\mres\Omega\right)=\nu
\end{align*}
if and only if it has the form
\begin{align*}
T=\sum_{i=1}^Nz_i\mathds{1}_{C_i},
\end{align*}
where $\{C_i\}_{i=1}^N$ is a tesselation of $\Omega$ by measurable sets $C_i$ such that
\begin{align*}
\mathcal{L}^3(C_i)=m_i\qquad \forall\, i\in\{1,...,N\}.
\end{align*}
Hence the optimal transport problem between $\mathcal{L}^3\mres\Omega$ and $\nu\in \mathcal{Q}^N(\R^3)$ is reduced to an optimal partitioning problem. Moreover, the unique optimal partition, and corresponding transport map, can be characterised using the notion of \emph{Laguerre tessellations}.
\begin{defn}[Laguerre tessellation]\label{def:LaguerreTess}
{\normalfont Given a \emph{seed vector} $\z=(z_1,...,z_N)\in\R^{3N}$ and a \emph{weight vector} $\w=(w_1,...,w_n)\in\R^N$, the \emph{Laguerre tessellation} of $\Omega$ generated by the pair $(\zw)$ is defined to be the family
\begin{align*}
\{C_i(\zw)\}_{i=1}^N,
\end{align*}
where $C_i(\zw)$ are \emph{Laguerre cells} defined by
\begin{align}\label{eqn:lagcell}
C_i(\zw)=\left\{x\in \Omega\ :\ |x-z_i|^2-w_i\leqslant |x-z_j|^2-w_j\quad \forall\, j\in\{1,...,N\}\right\}.
\end{align}}
\end{defn}
\noindent Note that any Laguerre cell is convex since it is the intersection of finitely many half-spaces with the convex set $\Omega$. In particular, Laguerre cells that do not intersect $\partial\Omega$ are polyhedra. Moreover, for any seed vector $\z$, weight vector $\w$ and indices $i\neq j$, the intersection $C_i(\zw)\cap C_j(\zw)$ is contained in the $2$-dimensional plane
\begin{align*}
\left\{x\in \R^3\ :\ |x-z_i|^2-w_i=|x-z_j|^2-w_j\right\}.
\end{align*}

Using the Kantorovich Duality Theorem (see, for example, \cite[Section 1.2]{santambrogio2015optimal}), one can show that the optimal transport cost between $\mathcal{L}^3\mres\Omega$ and $\nu$ is the supremum over all weight vectors $\w=(w_1,...,w_N)\in\R^N$ of the Kantorovich functional $g:\R^N\to\R$ defined by
\begin{align}\label{eq:g}
g(\w)=\sum_{i=1}^N\underset{C_i(\zw)}\int|x-z_i|^2\,\rd x+\sum_{i=1}^N\big(m_i-\mathcal{L}^3\left(C_i(\zw)\right)\big) w_i.
\end{align}
The Kantorovich functional $g$ is concave, and maximisers $\w\in\R^N$ of $g$ exist and satisfy
\begin{align}\label{eqn:OWVchar}
\mathcal{L}^3(C_i(\zw))=m_i\qquad\forall\, i\in\{1,...,N\}.
\end{align}
(See, for example, \cite[Theorem 40]{merigot2021optimal}.)
We call such $\w$ \emph{optimal weight vectors}. Note that $w_i=\psi(z_i)$ where $\psi$ is an optimal Kantorovich potential. Given an optimal weight vector $\w\in\R^N$, the unique optimal transport map from $\mathcal{L}^3\mres\Omega$ to $\nu$ is given by
\begin{align*}
T=\sum_{i=1}^Nz_i\mathds{1}_{C_i(\zw)}.
\end{align*}
In particular, if we define the function $\Phi:\Omega\to\R$ by
\begin{align}\label{eqn:Phi}
\Phi(x)=\frac{1}{2}|x|^2-\frac{1}{2}\underset{i}\min\{|x-z_i|^2-w_i\},
\end{align}
then $T=\nabla \Phi$ and, using the notation introduced in Definition \ref{def:BrenOper},
\begin{align}\label{eqn:B}
\mathcal{B}[\nu]=\Phi-\int_{\Omega}\Phi(x)\,\rd x.
\end{align}
This follows from the Gangbo-McCann Theorem (see, for example, \cite[Theorem 1.17]{santambrogio2015optimal}).

The Monge problem is a non-convex optimisation problem. As outlined above, in the case of semi-discrete optimal transport this can be replaced by an unconstrained, finite dimensional optimisation problem (maximising $g$), which is numerically tractable. As we demonstrate in Section \ref{sect:num}, this is one motivation for using semi-discrete optimal transport to construct solutions of SG in geostrophic coordinates.

\subsection{Optimal weight map}\label{sect:weight}
Optimal weight vectors are not unique. Indeed, let $\nu\in\mathcal{Q}^N(\R^3)$ have mass vector $\mathbf{m}$ and seed vector $\z$, and let $\mathbf{e}=(1,...,1)\in\R^N$. Using \eqref{eqn:lagcell} it is easy to see that for any $\lambda\in \R$, 
\begin{align*}
C_i(\zw)=C_i(\z,\w+\lambda\mathbf{e})\qquad \forall\, i\in\{1,...,N\}.
\end{align*}
Hence, by the characterisation of optimal weight vectors \eqref{eqn:OWVchar}, $\w\in\R^N$ is optimal if and only if every vector in $\w+\mathrm{span}\{\mathbf{e}\}$ is optimal. Conversely, if $\w,\,\widetilde{\w}\in\R^N$ and $\w-\widetilde{\w}\notin \mathrm{span}\{\mathbf{e}\}$, then the pairs $(\z,\w)$ and $(\z,\widetilde{\w})$ define distinct Laguerre tessellations, so at least one of $\w$ and $\widetilde{\w}$ is not optimal. In particular, it is easy to deduce that there is a unique optimal weight vector whose $N^{\text{th}}$ component is zero. This leads to the following definition, where $e_i$ denotes the $i^{\text{th}}$ canonical basis vector in $\R^N$.
\begin{defn}\label{def:OWM}
{\normalfont Given a fixed mass vector $\mathbf{m}\in\R^N$, we define the \emph{optimal weight map} $\mathbf{w}_*:D\to \R^N$ to be that which sends each seed vector $\mathbf{z}\in D$ to the unique optimal weight vector $\w_*(\z)\in \mathrm{span}\{e_1,...,e_{N-1}\}$. We refer to the family $\{C_i(\mathbf{z},\mathbf{w}_*(\mathbf{z}))\}_{i=1}^N$ as the \emph{optimal Laguerre tessellation} of $\Omega$ generated by $\mathbf{z}\in\R^N$.}
\end{defn}
\noindent Note that there are many possible definitions of the optimal weight map, which all yield the same definition of an optimal Laguerre tessellation. For instance, a natural choice of range space would be $\mathrm{span}\{\mathbf{e}\}^{\perp}$. We choose the range space $\mathrm{span}\{e_1,...,e_{N-1}\}$ so that subsequent arguments concerning the regularity of $\w_*$ can be carried out using only the canonical basis of Euclidean space.

\section{Statement of existence theorems}\label{sect:mainResults}

After stating some preliminary definitions, we state the two existence results (Theorem \ref{thm:1} and Theorem \ref{thm:2}) for which we give novel proofs. See Definition \ref{def:BrenOper} and equation \eqref{eqn:Q} for the definitions of the operator $\mathcal{B}$ and the space $\mathcal{Q}^N(\R^3)$, respectively.

\begin{defn}[Geostrophic energy]
{\normalfont The \emph{geostrophic energy} functional $E:\mathcal{P}_c(\R^3)\to \R$ is defined by
\begin{align*}
E[\nu]:=\int_{\Omega}\left(\frac{1}{2}\left((\partial_1\mathcal{B}[\nu](x)-x_1)^2+(\partial_2\mathcal{B}[\nu](x)-x_2)^2\right)-x_3\partial_3\mathcal{B}[\nu](x)\right) \,\rd x.
\end{align*}}
\end{defn}

\noindent Note that the geostrophic energy is traditionally written in terms of the geostrophic velocity field $u_{g}=(u_{g, 1}, u_{g, 2}, 0)^{T}$ and the potential temperature $\theta$ as
\begin{align*}
E(u_g,\theta)=\int_{\Omega}\left(\frac{1}{2}\left(u^2_{g,1}+u^2_{g,2}\right)-x_3\theta\right)\,\rd x.
\end{align*}

\begin{rem}\label{rem:energy}
{\normalfont The geostrophic energy functional is continuous on $\mathcal{P}(K)$ for any compact set $K\subset \R^3$. To see this first note that, for any $\nu\in\mathcal{P}_c(\R^3)$,
\begin{align}\label{eqn:energy}
E[\nu]=\frac{1}{2}W^2_2(\mathcal{L}^3\mres\Omega,\nu)
-\frac{1}{2}\|\partial_3\mathcal{B}[\nu]\|_{L^2(\Omega)}^2
-\frac{1}{2}\int_{\Omega}x^2_3\,\rd x.
\end{align}
Suppose that $(\nu^N)_{N\in\N}\subset\mathcal{P}(K)$ is a sequence which converges to a measure $\nu\in\mathcal{P}(K)$ as $N\to\infty$. Since $W_2$ is a metric on $\mathcal{P}(K)$,
\begin{align*}
\lim_{N\to\infty}W_2\left(\mathcal{L}^3\mres\Omega,\nu^N\right)=W_2\left(\mathcal{L}^3\mres\Omega,\nu\right)
\end{align*}
and, by continuity of $\mathcal{B}$ on $\mathcal{P}(K)$ (Theorem \ref{thm:quantStab}),
\begin{align*}
\lim_{N\to\infty}\|\partial_3\mathcal{B}[\nu^N]\|_{L^2(\Omega)}
=\|\partial_3\mathcal{B}[\nu]\|_{L^2(\Omega)}.
\end{align*}
Hence
\begin{align*}
\lim_{N\to\infty}E[\nu^N]=E[\nu].
\end{align*}}
\end{rem}

\begin{defn}[Geostrophic solution]\label{defn:SGG}
{\normalfont Let $T\in (0,\infty)$ and let $\overline{\alpha}\in\mathcal{P}_c(\R^3)$. We say that $\alpha\in C([0,T];\mathcal{P}_c(\R^3))$ is a weak solution of the 3D incompressible semi-geostrophic equations in geostrophic coordinates on $[0,T]$ with initial measure $\overline{\alpha}$ if
\begin{align}
\nonumber&\int_0^{T}\int_{\R^3}\left(\partial_t \varphi(z,t)+Jz\cdot\nabla \varphi(z,t)\right)\,\rd\alpha_t(z)\,\rd t-\int_0^{T}\int_{\Omega}J x\cdot\nabla \varphi\big(\nabla \mathcal{B}[\alpha_t](x),t\big)\,\rd x\,\rd t\\
\label{eqn:SGGdist}&\quad\qquad\qquad\qquad\qquad\qquad\qquad\qquad\qquad=\int_{\R^3}\varphi(z,T)\,\rd\alpha_{T}(z)-\int_{\R^3}\varphi(z,0)\,\rd\overline{\alpha}(z),
\end{align}
for all $\varphi\in \mathscr{D}(\R^3\times \R)$, where the matrix  $J$ is defined by \eqref{eqn:J}.
In what follows, we will refer to such a map $\alpha$ as a \emph{geostrophic solution}. In particular, if there exists $N\in\N$, a ($k$-times continuously differentiable) map $\z=(z_1,...,z_N):[0,T]\to\R^{3N}$ and a mass vector $\mathbf{m}\in\R^N$ such that
\begin{align}\label{eqn:ansatz}
\alpha_t=\sum_{i=1}^Nm_i\delta_{z_i(t)}\in \mathcal{Q}^N(\R^3)
\end{align}
for all $t\in[0,T]$, we will refer to $\alpha$ as a ($k$-times continuously differentiable) \emph{discrete geostrophic solution}. If the corresponding geostrophic energy $E[\alpha_t]$ is constant in time, then we say that $\alpha$ is \emph{energy-conserving}.}
\end{defn}

\noindent The defintion of a \emph{geostrophic solution} of SG coincides with Loeper's definition of a \emph{weak measure solution} \cite[Definition 2.2]{loeper2006fully}. Moreover, if $\alpha$ is a geostrophic solution and $\alpha_t\ll \mathcal{L}^3$ for all $t\in[0,T]$ then, since 
\begin{align*}
\nabla \mathcal{B}[\alpha_t]^*\#\alpha_t=\mathcal{L}^3\mres\Omega,
\end{align*}
by a change of variables,
\begin{align*}
\int_{\Omega}Jx\cdot\nabla \varphi\big(\nabla \mathcal{B}[\alpha_t](x),t\big)\,\rd x=\int_{\R^3}J\nabla \mathcal{B}[\alpha_t]^*(z)\cdot\nabla \varphi(z,t)\, \rd \alpha_t(z)\qquad \forall\, \varphi\in \mathscr{D}(\R^3\times\R)
\end{align*}
for all $t\in[0,T]$. Hence $\alpha$ is a distributional solution of the active transport equation
\begin{align}\label{eqn:GeoActTra}
\partial_t\alpha+J\big(\text{id}_{\R^3}-\nabla\mathcal{B}[\alpha]^*\big)\cdot\nabla \alpha=0.
\end{align}
This is the setting considered in \cite{benamou1998weak}.

\begin{defn}[Well-preparedness]\label{defn:wellPrep}
{\normalfont Let $N\in\N$. We say that a discrete probability measure $\beta\in \mathcal{Q}^N(\R^3)$ given by
\begin{align*}
\beta=\sum_{i=1}^Nm_i\delta_{z_i}
\end{align*}
is \emph{well-prepared} for SG in geostrophic coordinates if there exists $r>0$ such that
\begin{align}\label{eqn:wellPrep1}
|(z_i-z_j)\cdot e_3|\geqslant r\quad \forall\, i\neq j.
\end{align}
In other words, $\beta$ is well prepared if the seeds $z_i$ lie in distinct horizontal planes.}
\end{defn}

Well-preparedness of the initial data ensures that the seeds do not collide and therefore do not enter the set on which the right-hand side of the ODE in \eqref{eqn:ODEIVPintro} is discontinuous. This allows us to prove the following theorem.

\begin{thm}[{\normalfont Existence of discrete geostrophic solutions, c.f. \cite[Proposition 4.14]{feldman2013lagrangian}}]\label{thm:1}
Let $\Omega\subset \R^3$ be open, bounded and convex, and fix $N\in\N$, $N\geq 2$. For any $T\in (0,\infty)$ and any well-prepared discrete probability measure $\overline{\alpha}\in \mathcal{Q}^N(\R^3)$, there exists a unique twice continuously differentiable discrete geostrophic solution $\alpha\in C([0,T];\mathcal{P}_c(\R^3))$ with initial measure $\overline{\alpha}$. Moreover, this solution is energy-conserving.
\end{thm}
\noindent In Example \ref{exa:singleMass}, we show that Theorem \ref{thm:1} extends easily to the case $N=1$ by deriving an explicit expression for the solution.

The uniqueness and regularity of solutions attained in Theorem \ref{thm:1} are significant in the context of numerical analysis since these factors determine the convergence of the corresponding numerical method. As mentioned in the introduction, the result of Theorem \ref{thm:1} is analogous to \cite[Proposition 4.14]{feldman2013lagrangian}, but with the following differences. In \cite[Proposition 4.14]{feldman2013lagrangian}, the physical domain $\Omega$ need not be convex, and the initial measure can be any convex combination of Dirac masses. However, with these slightly weaker hypotheses, the seed trajectories $\mathbf{z}$ are only known to be Lipschitz in time and are not known to be unique.

\begin{thm}[Existence of geostrophic solutions]\label{thm:2}
Let $\Omega\subset \R^3$ be open, bounded and convex. For any $T\in (0,\infty)$ and any $\overline{\alpha}\in\mathcal{P}_c(\R^3)$, there exists an energy-conserving geostrophic solution $\alpha\in C^{0,1}([0,T];\mathcal{P}_c(\R^3))$ with initial measure $\overline{\alpha}$ and a sequence $(\alpha^N)_{N\in\N}$ of twice continuously differentiable discrete geostrophic solutions which converges uniformly in $C([0,T];\mathcal{P}_c(\R^3))$ to $\alpha$:
\begin{align*}
\lim_{N\to\infty}\sup_{t\in[0,T]}W_2(\alpha^N_t,\alpha_t)=0.
\end{align*}
\end{thm}

\noindent The existence of geostrophic solutions with arbitrary compactly supported initial measure was first proved by Loeper \cite[Theorem 2.3]{loeper2006fully} and later by Feldman and Tudorascu in the context of \emph{weak Lagrangian} solutions (see \cite[Theorem 3.2]{feldman2015semi}). Being uniform in time, as opposed to pointwise in time, the convergence obtained in Theorem \ref{thm:2} is stronger than that obtained in these previous works. Since we work only with compactly supported measures, the spatial convergence obtained in Theorem \ref{thm:2} is equivalent to that obtained in \cite{feldman2015semi}. As in \cite{feldman2015semi}, the limit point obtained in Theorem \ref{thm:1} is Lipschitz in time. Note that the corresponding theorems in \cite{feldman2015semi} and \cite{loeper2006fully} do not include the hypothesis that $\Omega$ is convex. We include this hypothesis for technical reasons relating to the regularity of the \emph{optimal centroid map} (see Definition \ref{def:centroid} and Remark \ref{rem:nonconvex}).

\begin{rem}[Conservation of transport cost]\label{rem:ConsTransCost}
{\normalfont  From the proof of Lemma \ref{lem:discEnergyCons} it is easy to deduce that each discrete solution $\alpha^N$ conserves not only the geostrophic energy but also the transport cost $W_2(\alpha^N_t,\mathcal{L}^3\mres\Omega)$. Therefore any solution $\alpha$ constructed as the uniform limit of discrete solutions also conserves the transport cost $W_2(\alpha_t,\mathcal{L}^3\mres\Omega)$ by continuity of the Wasserstein distance. This is essentially due to the Hamiltonian structure of equation \eqref{eqn:SGgeoActTrans} (see \cite[Example 8.1(c)]{ambrosio2008hamiltonian}) and can be seen as a special case of \cite[Theorem 5.2]{ambrosio2008hamiltonian}. Conservation of the transport cost for weak Lagrangian solutions is proved in \cite[Corollary 5.2]{feldman2013lagrangian}. For clarity, we include an independent proof based on our semi-discrete solution procedure.}
\end{rem}

\section{Proof of Theorem \ref{thm:1}}\label{sect:discExist}

\noindent Fix $N\in\N, N\geqslant 2$. (The case $N=1$ is discussed in Example \ref{exa:singleMass}.) To construct a twice continuously differentiable discrete geostrophic solution with well-prepared initial measure $\overline{\alpha}\in\mathcal{Q}^N(\R^3)$, we substitute the expression \eqref{eqn:ansatz} into the transport equation \eqref{eqn:SGGdist} and, by appropriate choice of test functions, derive an ODE-IVP for the paths $\mathbf{z}=(z_1,...,z_N)$. We then prove in Proposition \ref{prop:ODEex} that this ODE-IVP has a unique $C^2$-solution. Conversely, we show that any $C^1$-solution of the ODE-IVP gives rise to an energy-conserving discrete geostrophic solution via the formula \eqref{eqn:ansatz}, from which Theorem \ref{thm:1} follows. The most involved part of the proof of Theorem \ref{thm:1} is the proof of the regularity of the \emph{optimal centroid map}, which we now define.

\begin{defn}\label{def:centroid}
{\normalfont Define the set
\begin{align*}
\widetilde{D}:=\left\{(\z,\w)\in \R^{3N}\times\R^N\ \big\vert\ \z\in D,\ \mathcal{L}^3(C_i(\z,\w))>0\ \forall\,i\in\{1,...,N\}\right\},
\end{align*}
where $D$ is the set of seed vectors defined by \eqref{eqn:D}. $\widetilde{D}$ is the set of generators of Laguerre tessellations of $\Omega$ with no empty cells. We define the \emph{centroid map} $\overline{\mathbf{x}}:\widetilde{D}\to \Omega^N$, $\overline{\mathbf{x}}=\left(\overline{x}_1,...,\overline{x}_N\right)$ by
\begin{align*}
\overline{x}_i(\zw)=\frac{1}{\mathcal{L}^3\left(C_i(\zw)\right)}\underset{C_i(\zw)}\int x\,\rd x,\qquad i\in\{1,...,N\}.
\end{align*}
Moreover, given a fixed mass vector $\mathbf{m}$, we define the \emph{optimal centroid map} $\mathbf{x}=\left(x_1,...,x_N\right):D\to \Omega^N$ by
\begin{align*}
\mathbf{x}(\z):=\overline{\mathbf{x}}(\z,\w_*(\z)),
\end{align*}
where $\w_*$ is the optimal weight map (Definition \ref{def:OWM}).}
\end{defn}
\noindent 
We now characterise $k$-times continuously differentiable discrete geostrophic solutions in terms of solutions of an ODE-IVP involving the optimal centroid map.

\begin{lem}\label{lem:odeChar}
Let $\overline{\alpha}\in\mathcal{Q}^N(\R^3)$ be given by
\begin{align}\label{eqn:discIC}
\overline{\alpha}=\sum_{i=1}^Nm_i\delta_{\oz_i}.
\end{align}
A map $\alpha:[0,T]\to\mathcal{Q}^N(\R^3)$ given by
\begin{align}\label{eqn:ansatzLem}
\alpha_t=\sum_{i=1}^Nm_i\delta_{z_i(t)}\quad \forall\,t\in[0,T]
\end{align}
is a $k$-times continuously differentiable discrete geostrophic solution with initial measure $\overline{\alpha}$, for $k\in\N$, if and only if the map $\z=(z_1,...,z_N):[0,T]\to\R^{3N}$ is a $k$-times continuously differentiable solution of the ODE-IVP
\begin{align}\label{eqn:ODE-IVP}
\begin{cases}
\ \dot{\mathbf{z}}=W(\mathbf{z}),\\
\ \mathbf{z}(0)=\overline{\mathbf{z}},
\end{cases}
\end{align}
where $\obz=(\oz_1,...,\oz_N)$,
\begin{align}\label{eqn:W}
W(\mathbf{z}):=J_N(\mathbf{z}-\mathbf{x}(\mathbf{z})),
\end{align}
and $J_N\in \R^{3N\times 3N}$ is the block diagonal matrix
\begin{align*}
J_N:=\mathrm{diag}(J,...,J).
\end{align*}
\end{lem}

\begin{proof}
Suppose that $\z=(z_1,...,z_N):[0,T]\to\R^{3N}$ is a $k$-times continuously differentiable solution of \eqref{eqn:ODE-IVP} and let $\alpha$ be given by \eqref{eqn:ansatzLem}. Since $\z$ is continuous, $\alpha\in C([0,T];\mathcal{P}_c(\R^3))$. We must check that $\alpha$ satisfies \eqref{eqn:SGGdist}. Letting $\varphi\in \mathscr{D}(\R^3\times\R)$ be arbitrary, we have
\begin{align*}
&\int_0^{T}\int_{\R^3}\left(\partial_t \varphi(z,t)+Jz\cdot\nabla \varphi(z,t)\right)\,\rd\alpha_t(z)\,\rd t-\int_0^{T}\int_{\Omega}J x\cdot\nabla \varphi\big(\nabla \mathcal{B}[\alpha_t](x),t\big)\,\rd x\,\rd t\\
&\quad=\sum^N_{i=1}m_i\int_0^{T}\left(\partial_t\varphi(z_i(t),t)+J\big(z_i(t)-{x}_i(\mathbf{z}(t))\big)\cdot\nabla \varphi(z_i(t),t)\right)\,\rd t\\
&\quad=\sum^N_{i=1}m_i\int_0^{T}\left(\partial_t\varphi(z_i(t),t)+\dot{z}_i(t)\cdot\nabla \varphi(z_i(t),t)\right)\,\rd t\\
&\quad=\sum^N_{i=1}m_i\int_0^{T}\done{}{t}\varphi(z_i(t),t)\,\rd t\\
&\quad=\sum^N_{i=1}m_i\varphi(z_i(T),T)-\sum^N_{i=1}m_i\varphi(z_i(0),0)\\
&\quad=\int_{\R^3}\varphi(z,T)\,\rd\alpha_{T}(z)-\int_{\R^3}\varphi(z,0)\,\rd\overline{\alpha}(z),
\end{align*}
as required.

Conversely, suppose that $\alpha:[0,T]\to\mathcal{Q}^N(\R^3)$ given by \eqref{eqn:ansatzLem} is a $k$-times continuously differentiable discrete geostrophic solution with initial measure $\overline{\alpha}$, in the sense of Definition \ref{defn:SGG}. Then substitution of \eqref{eqn:ansatzLem} into \eqref{eqn:SGGdist} yields
\begin{align}
\nonumber &\sum^N_{i=1}m_i\int_0^{T}\left(\partial_t\varphi(z_i(t),t)+J\big(z_i(t)-{x}_i(\z(t))\big)\cdot\nabla \varphi(z_i(t),t)\right)\,\rd t\\
\label{eqn:wSGgeodisc}&\qquad\qquad\qquad\qquad\qquad\qquad\qquad\qquad\qquad\qquad=\sum_{i=1}^Nm_i\Big(\varphi\big(z_i(T),T\big)-\varphi(\overline{z}_i,0)\Big)
\end{align}
for all $\varphi\in \mathscr{D}(\R^3\times\R)$. We now show that the paths $z_i$ can be separated, and deduce that the ODE-IVP \eqref{eqn:ODE-IVP} is satisfied by $\z=(z_1,...,z_N)$ by choosing test functions which isolate each path. Since $\alpha_t\in \mathcal{Q}^N(\R^3)$ for each $t\in [0,T]$, and $\z$ is continuous on $[0,T]$ by hypothesis, the map
\begin{align*}
t\mapsto d(t):=\min_{i\neq j}|z_i(t)-z_j(t)|
\end{align*}
is positive and continuous on the compact interval $[0,T]$ and, therefore, attains a positive minimum on $[0,T]$. In other words, by hypothesis, there exists $d_*>0$ such that
\begin{align*}
\min_{t\in[0,T]}\min_{i\neq j}|z_i(t)-z_j(t)|\geqslant d_*.
\end{align*}
Hence, for fixed $t_0\in (0,T)$ and fixed $i\in\{1,...,N\}$, there exists an open interval $I\subset (0,T)$ containing $t_0$ and an open set $U_i\subset\R^3$ such that
\begin{align*}
z_i(I)\subset U_i\qquad \text{and}\qquad \bigcup_{j\neq i}z_j(I)\cap U_i=\emptyset.
\end{align*}
For a fixed coordinate index $k\in\{1,2,3\}$, consider a test function $\varphi_{i,k}\in \mathscr{D}(\R^3\times\R)$ of the form $\phi\psi_{i,k}$, where $\phi\in C^{\infty}_c(I)$ and $\psi_{i,k}\in C^{\infty}_c(\R^3)$ satisfies
\begin{align*}
\psi_{i,k}(z)=z\cdot e_k \quad \forall\,z \in U_i,\qquad\text{and}\qquad
\bigcup_{j\neq i}z_j(I)\cap\mathrm{supp}(\psi_{i,k})=\emptyset.
\end{align*}
By the Chain Rule, for all $t\in I$,
\begin{align*}
\partial_t\varphi_{i,k}(z_j(t),t)
&=\begin{cases}
\done{}{t}\big(\phi(t)\psi_{i,k}(z_i(t))\big)-\phi(t)\dot{z}_i(t)\cdot e_k \quad &\text{if }j=i,\\
0 &\text{if }j\neq i.
\end{cases}
\end{align*}
Hence, with this particular choice of test function, equation \eqref{eqn:wSGgeodisc} becomes
\begin{align*}
\int_I m_i\done{}{t}\big(\phi(t)\psi_{i,k}(z_i(t))\big)\,\rd t
+\int_I m_i\left(-\dot{z}_i(t)+J\left(z_i(t)-x_i(\mathbf{z}(t))\right)\right)\cdot e_k\phi(t)\,\rd t=0.
\end{align*}
Since $\phi\in C^{\infty}_c(I)$, the first integral is zero. By varying over all $k\in\{1,2,3\}$ and all $\phi\in C^{\infty}_c(I)$ we obtain the pointwise statement that
\begin{align*}
\dot{z}_i(t)=J\left(z_i(t)-x_i(\mathbf{z}(t))\right)
\end{align*}
for all $t\in I$, in particular for $t=t_0$. In a similar way, it can be shown that the inital condition $z_i(0)=\oz_i$ is satisfied. Repeating this for all $t_0\in (0,T)$, and all $i\in\{1,...,N\}$, we deduce that $\mathbf{z}=(z_1,...,z_N)$ satisfies the ODE-IVP \eqref{eqn:ODE-IVP}.
\end{proof}

\begin{rem}\label{rem:discActTra}
{\normalfont The ODE-IVP \eqref{eqn:ODE-IVP} is precisely the discrete analogue of the active transport equation \eqref{eqn:SGgeoActTrans}. Indeed, let $\nu\in\mathcal{Q}^N(\R^3)$ have seed vector $\z=(z_1,...,z_N)$ and consider the extension by $+\infty$ of the convex function $\mathcal{B}[\nu]$ to $\R^3$, which we will again denote by $\mathcal{B}[\nu]$. For each $i\in\{1,...,N\}$, $x_i(\mathbf{z})$ is the centroid of the subdifferential of $\mathcal{B}[\nu]^*$ evaluated at the point $z_i$.
To see this, first recall that for any convex function $f:\Omega\to\R$ and any constant $c\in\R$
\begin{align*}
\partial f^*(z)=\partial (f+c)^*(z)\qquad \forall\, z\in\R^3.
\end{align*}
Therefore, the characterisation of $\mathcal{B}[\nu]$ given by \eqref{eqn:Phi} and \eqref{eqn:B} implies that, for each $i\in\{1,...,N\}$,
\begin{align*}
\partial(\mathcal{B}[\nu])^*(z_i)=\partial\left(\Phi-\int_{\Omega}\Phi\,\rd x\right)^*(z_i)=\partial \Phi^*(z_i),
\end{align*}
where, letting $\w_*(\z)=(w_1,...,w_N)$,
\begin{align*}
\Phi(x)=\begin{cases}
\frac{1}{2}|x|^2-\frac{1}{2}\underset{j}\min\left\{|x-z_j|^2-w_j\right\}\qquad \text{if}\ x\in\Omega,\\
+\infty \qquad\qquad\qquad\qquad\qquad\qquad\quad\hspace{4pt} \text{if}\ x\notin\Omega.
\end{cases}
\end{align*}
Note that
\begin{align*}
\Phi^*(z_i)=\frac{1}{2}\left(|z_i|^2-w_i\right) \qquad \forall\, i\in\{1,...,N\}.
\end{align*}
By the characterisation of the subdifferential of $\Phi$ in terms of its Legendre-Fenchel transform (see, for example, \cite[Proposition 2.4]{villani2003topics}), we have
\begin{align*}
\partial \Phi^*(z_i)&=\left\{x\in\Omega\ :\ \Phi(x)+\Phi^*(z_i)=x\cdot z_i\right\}\\
&=\left\{x\in\Omega\ :\ |x-z_i|^2-w_i=\min_{j}\left\{|x-z_j|^2-w_j\right\}\right\}\\
&=C_i(\z,\w_*(\z)).
\end{align*}
Hence
\begin{align*}
\partial \left(\mathcal{B}[\nu]\right)^*(z_i)=C_i(\z,\w_*(\z)),
\end{align*}
as required. This is consistent with the `geometric interpretation' of a geostrophic solution described in \cite[p.803]{loeper2006fully}.}
\end{rem}

Having characterised $k$-times continuously differentiable discrete geostrophic solutions in terms of solutions of the ODE-IVP \eqref{eqn:ODE-IVP}, we now aim to show that solutions of \eqref{eqn:ODE-IVP} exist and that the corresponding geostrophic solutions are energy-conserving. After proving some preliminary results (Lemma \ref{lem:ape} and Lemma \ref{lem:diffOWM}), we prove the following Proposition.

\begin{prop}\label{prop:ODEex}
For $\obz=(\oz_1,...,\oz_N)\in \R^{3N}$ satisfying
\begin{align}\label{eqn:wellPrep}
\min_{i\neq j}\vert(\oz_i-\oz_j)\cdot e_3\vert\geqslant r
\end{align}
for some $r>0$, there exists a unique $C^2$-solution of \eqref{eqn:ODE-IVP} on the interval $[0,T]$ for any $T\in(0,\infty)$.
\end{prop}

By Lemma \ref{lem:odeChar}, this gives rise to a unique discrete geostrophic solution with initial measure \eqref{eqn:discIC} via the formula \eqref{eqn:ansatzLem}. As we will see in Section \ref{sect:generalExistence}, the condition \eqref{eqn:wellPrep} on the initial data is not restrictive for the purpose of constructing a geostrophic solution with arbitrary initial measure in $\mathcal{P}_c(\R^3)$. We begin by explaining its relevance.

We show below that the map $W$, defined by \eqref{eqn:W}, is continuously differentiable, and therefore locally Lipschitz on $D$. However, as we demonstrate in Remark \ref{rem:noCtsExt}, $W$ does not, in general, admit continuous extension to $\R^{3N}$. In order to apply the Picard-Lindel\"{o}f Existence Theorem to obtain a unique solution of \eqref{eqn:ODE-IVP} on a given time interval $[0,T]$, it is therefore necessary to ensure that any solution trajectory is bounded away from the boundary of $D$. This is guaranteed a priori if $\overline{\z}$ satisfies \eqref{eqn:wellPrep}. Indeed, since the third row of the matrix $J$ is zero, if $\obz\in\R^{3N}$ satisfies \eqref{eqn:wellPrep} and $\z$ is a solution of the corresponding ODE-IVP \eqref{eqn:ODE-IVP}, then $\z(t)$ satisfies \eqref{eqn:wellPrep} for all times $t$.

We now prove the claimed regularity of the map $W$ in the case where the domain $\Omega$ is open, bounded and convex (see Remark \ref{rem:nonconvex} for a discussion of the case where $\Omega$ is non-convex). To do so, we use the following results regarding the regularity of the centroid map $\obx$ (see Definition \ref{def:centroid}) and the volume map $\mathbf{V}$ (see Definition \ref{defn:vol}), as well as the structure of the  matrix of partial derivatives of $\mathbf{V}$ with respect to $\w$, which can be described using the notions of a \emph{Laplacian matrix} and the \emph{dual graph} of a Laguerre tesselation.

\begin{defn}[Dual graph]\label{defn:dualGraph}
{\normalfont Given a Laguerre tessellation $\{C_i\}_{i=1}^N$ of $\Omega$, for each $i\in\{1,...,N\}$ the set of neighbours of $i$ is defined to be the set
\begin{align*}
N_i:=\left\{j\in\{1,...,N\}\ \vert\ j\neq i,\ C_i\cap C_j\neq \emptyset\right\}.
\end{align*}
The (undirected) graph $(V,E)$ given by
\begin{align*}
V=\{1,...,N\},\qquad E=\big\{\{i,j\}\subset\{1,...,N\}:j\in N_i\big\},
\end{align*}
is referred to as the \emph{dual graph} of the Laguerre decomposition $\{C_i\}_{i=1}^N$.}
\end{defn}
\begin{defn}[Laplacian matrix \cite{mohar1997some}]\label{defn:laplacianMatrix}
{\normalfont Given a weighted graph $G=(V,E,h)$, where $V=\{1,...,N\}$ and $h:E\to\R$, the adjacency matrix of $G$, $A_G\in\R^{N\times N}$, is given by
\begin{align*}
(A_G)_{ij}=\begin{cases}
h(\{i,j\})\quad &\text{if}\quad \{i,j\}\in E,\\
0\quad &\text{otherwise}.
\end{cases}
\end{align*}
The degree matrix of $G$, $D_G\in\R^{N\times N}$, is the diagonal matrix such that
\begin{align*}
(D_G)_{ii}=\sum\limits_{j=1}^N(A_G)_{ij}.
\end{align*}
The \emph{Laplacian matrix} of $G$, $L_G\in\R^{N\times N}$, is then given by
\begin{align*}
L_G=D_G-A_G.
\end{align*}}
\end{defn}

\begin{thm}[{\normalfont c.f. \cite[Proposition 2 and Lemma 2]{de2019differentiation}, \cite[Theorem 4.1]{kitagawa2016convergence}, \cite[Theorem 45]{merigot2021optimal}}]\label{thm:diffFctnal}
Let $f\in C(\Omega)\cap W^{1,1}(\Omega)$ and define $F=(F_1,...,F_N):\widetilde{D}\to\R^{N}$ by
\begin{align*}
F_i(\zw):=\int\limits_{C_i(\zw)}f(x)\,\rd x.
\end{align*}
Then $F$ is continuously differentiable. In particular, for $(\zw)\in\widetilde{D}$ let $G=(V,E)$ denote the dual graph of the Laguerre tessellation $\{C_i(\zw)\}_{i=1}^N$ and define $h:E\to\R$ by
\begin{align*}
h(\{i,j\})=\frac{1}{2|z_j-z_i|}\int_{C_i\cap C_j}f(x)\,\rd \mathcal{H}^2(x).
\end{align*}
Then the matrix $D_{\w}F(\zw)$ of partial derivatives of $F$ with respect to $\w$ evaluated at $(\zw)$ is the Laplacian matrix of the weighted graph $G=(V,E,h)$.
\end{thm}
Note that the expressions for the partial derivatives of $F$ with respect to the weights given above differs by a factor of $2$ from those given in \cite{de2019differentiation} due to our choice of quadratic cost function. Note also that the assumption that the seed locations are \emph{generic} with respect to the cost, which is used in the proof of \cite[Theorem 45]{merigot2021optimal}, is not needed for the quadratic cost if the cells are non-empty.
\begin{defn}\label{defn:vol}
{\normalfont We define the \emph{volume map} $\mathbf{V}:D\times \R^N\to\R^N,\ \mathbf{V}=(V_1,...,V_N),$ by
\begin{align*}
V_i(\zw)=\mathcal{L}^3\left(C_i(\zw)\right).
\end{align*}}
\end{defn}

By combining Theorem \ref{thm:diffFctnal} with the quotient rule for derivatives, we immediately obtain the following corollary.

\begin{cor}[{\normalfont c.f. \cite[Proposition 2]{de2019differentiation}, \cite[Theorem 4.1]{kitagawa2016convergence}, \cite[Lemma 2.4]{bourne2015centroidal}}]\label{cor:diffVx}
The volume map $\mathbf{V}$ and the centroid map $\obx$ (see Definition \ref{def:centroid}) are continuously differentiable on $\widetilde{D}$. In particular, for $(\zw)\in\widetilde{D}$, let $G=(V,E)$ denote the dual graph of the Laguerre tessellation $\{C_i(\zw)\}_{i=1}^N$ and define $h:E\to\R$ by
\begin{align*}
h(\{i,j\})=\frac{\mathcal{H}^2(C_i\cap C_j)}{2|z_j-z_i|},
\end{align*}
where $\mathcal{H}^2$ denotes the 2-dimensional Hausdorff measure on $\R^3$. Then the matrix $D_{\w}V(\zw)$ of partial derivatives of $F$ with respect to $\w$ evaluated at $(\zw)$ is the Laplacian matrix of the weighted graph $G=(V,E,h)$.
\end{cor}

The optimal centroid map $\mathbf{x}$ (see Definition \ref{def:centroid}) is the composition of the centroid map $\obx$ with the optimal weight map $\w_*$ (see Definition \ref{def:OWM}). To show that $\mathbf{x}$ is continuously differentiable we now prove that $\w_*$ is continuously differentiable. At the time of writing, we believe this result to be novel.

\begin{lem}\label{lem:diffOWM}
Given a fixed mass vector $\mathbf{m}$, the corresponding optimal weight map $\w_*:D\to \R^N$ is continuously differentiable.
\end{lem}

\begin{proof}
Let $A\in\mathbb{R}^{N\times (N-1)}$ be the matrix
\begin{align*}
A=\left(\begin{array}{c}
 \mathbf{I}_{N-1}\\
\mathbf{0}
\end{array}\right),
\end{align*}
where $\mathbf{I}_{N-1}$ denotes the $(N-1)\times(N-1)$ identity matrix. Define
the residual map $\mathbf{r}:D\times\mathbb{R}^{N-1}\to\mathbb{R}^{N-1}$ by
\begin{align*}
\mathbf{r}(\zw)=A^{\mathrm{T}} \left( \mathbf{V}(\mathbf{z},A \mathbf{w})-\mathbf{m} \right).
\end{align*}
Observe that $AA^{\mathrm{T}}\w_*=\w_*$ so for all $\z\in D$ 
\begin{align*}
\mathbf{r}(\z,A^{\mathrm{T}}\w_*(\z))=0.
\end{align*}
Now fix an arbitrary seed vector $\z\in D$. Since $m_i>0$ for each $i\in\{1,...,N\}$, $(\z,\w_*(\z))\in \widetilde{D}$. By Corolloary \ref{cor:diffVx}, $\mathbf{V}$ is therefore continuously differentiable in a neighbourhood of $(\z,\w_*(\z))$. The residual map $\mathbf{r}$ is therefore continuously differentiable in a neighbourhood of $(\z,A^{\mathrm{T}}\w_*(\z))$ and
\begin{align*}
D_{\w}\mathbf{r}(\z,A^{\mathrm{T}}\w_*(\z))=A^{\mathrm{T}}D_{\w}V(\z,\mathbf{w}_*(\z))A.
\end{align*}
The weighted Laplacian matrix $D_{\w}V(\z,\mathbf{w}_*(\z))$ has positive entries and corresponds to a connected graph. Hence it is symmetric, positive semi-definite and its kernel is $\mathrm{span}\{\mathbf{e}\}$, where $\mathbf{e}=(1,...,1)\in\R^N$ (see for example \cite[Section 2.4]{mohar1997some}). The symmetric matrix $L=D_{\w} \mathbf{r}(\z,A^{\mathrm{T}}\w_*(\z))$ is therefore invertible. Indeed, if $\mathbf{y} \in \mathbb{R}^{N-1}$,
\begin{align*}
 \mathbf{y}^T L  \mathbf{y} = 0  \quad \Longleftrightarrow \quad
 \mathbf{y}^T A^{\mathrm{T}} D_{\w} \mathbf{V}(\z,\w_*(\z)) A \mathbf{y} = 0
\quad \Longleftrightarrow \quad
 A \mathbf{y} \in \mathrm{span}\{\mathbf{e}\} \quad \Longleftrightarrow \quad \mathbf{y} = \mathbf{0}.
\end{align*}
By the Implicit Function Theorem (see, for example, \cite[Theorem 10.2.1, p.270]{dieudonne1969foundations}) applied to $\mathbf{r}$ at the point $(\z,A^{\mathrm{T}}\w_*(\z))$, the function $A^{\mathrm{T}}\w_*:D\to\R^{N-1}$ is therefore continuously differentiable in a neighbourhood of $\z$. Since $AA^{\mathrm{T}}\w_*=\w_*$, it follows that the optimal weight map $\w_*$ is continuously differentiable in a neighbourhood of $\z$, as required.
\end{proof}

\begin{rem}[Counterexample 1]\label{rem:noCtsExt}
{\normalfont Combining Corollary \ref{cor:diffVx} with Lemma \ref{lem:diffOWM} we see that the optimal centroid map $\mathbf{x}$ is continuously differentiable on its domain $\widetilde{D}$. However, it does not, in general, admit continuous extension to $\R^{3N}$. We demonstrate this by means of a simple counter example in the case $N=2$. Let $\Omega\subset \R^3$ be the ball of volume $1$ centred at the origin, fix the mass vector $\mathbf{m}=(1/2,1/2)$, and consider the map $\z=(z_1,z_2):(-1,1)\to \R^6$ given by
\begin{align*}
z_1(s)=(s,0,0),\qquad z_2(s)=-z_1(s)\qquad \forall\,s\in (-1,1).
\end{align*}
Letting $x_*$ denote the centroid of the spherical cap
\begin{align*}
C:=\{x\in \Omega\ :\ x_1\geqslant 0\},
\end{align*}
it can easily be shown that the corresponding optimal centroids are given by
\begin{align*}
x_1(\z(s))=\begin{cases}
x_*\quad \text{if}\ s\in (0,1)\\
-x_*\quad \text{if}\ s\in (-1,0)
\end{cases}
\end{align*}
and
\begin{align*}
x_2(\z(s))=-x_1(\z(s))\quad \forall\, s\in (-1,0)\cup(0,1).
\end{align*}
Since $x_*\neq 0$, the map $\mathbf{x}\circ \z=\left(x_1\circ\z,x_2\circ\z\right)$ does not admit a continuous extension to the whole interval $(-1,1)$. Since $\z$ is continuous on $(-1,1)$, this means that the optimal centroid map $\mathbf{x}$ does not admit a continuous extension at the point $(0,0)\in\R^6$, precisely where the two seeds are the same.}
\end{rem}

\begin{rem}[Counterexample 2]\label{rem:nonconvex}
{\normalfont Corollary \ref{cor:diffVx} does not hold in general if $\Omega$ is non-convex. Indeed, let $\Omega\subset \R^3$ be the domain $(0,1)^3\setminus[1/2,1)^3$, let $\w=(0,0)$ and consider seed vectors given by the map $\z=(z_1,z_2):(-1/2,1/2)\to\R^6$ defined by
\begin{align*}
z_1(s)=(0,0,s),\qquad z_2(s)=(0,0,1+s).
\end{align*}
For $s\in (-1/2,1/2)$, the cell $C_1(\z(s),\w)$ is then given by
\begin{align*}
C_1(\z(s),\w)=\left\{x\in\Omega\ :\ x\cdot e_3\leq \frac{1}{2}+s\right\},
\end{align*}
and has volume
\begin{align*}
V_1(\z(s),\w)=
\begin{cases}
\frac{1}{2}+s\quad \text{if}\ s\in(-\frac{1}{2},0]\\
\frac{1}{2}+\frac{3}{4}s\quad \text{if}\ s\in(0,\frac{1}{2}).
\end{cases}
\end{align*}
The map $s\mapsto V_1(\z(s),\w)$ is not differentiable at $s=0$ which, since $\z$ is differentiable, implies that $V_1$ is not differentiable at the point $(\z(0),\w)$.

If $\Omega$ is open, bounded and convex, we can use Corollary \ref{cor:diffVx} to deduce that the optimal centroid map $\mathbf{x}$ is continuously differentiable and, therefore, locally Lipschitz. While it is known that the volume map $\mathbf{V}$ is Lipschitz in $\w$ even when the domain $\Omega$ is non-convex (see, for example \cite[Proposition 41]{merigot2021optimal}), to prove that $\mathbf{x}$ is locally Lipschitz for non-convex domains would require a finer analysis of the regularity of the centroid map $\obx$ and the volume map $\mathbf{V}$ with respect to $\z$.}
\end{rem}

We now prove that for any finite time horizon $T$, solutions of the ODE-IVP \eqref{eqn:ODE-IVP} on the interval $[0,T]$ are bounded and have bounded first derivatives. Our immediate application of these a priori estimates is to prove that such solutions exist (Proposition \ref{prop:ODEex}). As we show in Section \ref{sect:generalExistence}, these estimates also yield the necessary compactness of the corresponding sequence of measure-valued maps to pass to the limit as $N\to\infty$.

\begin{lem}[A priori estimates]\label{lem:ape}
If $\mathbf{z}=(z_1,...,z_N)$ is a $C^1$-solution of \eqref{eqn:ODE-IVP} on the interval $[0,T]$ then, for each $i\in\{1,...,N\}$,
\begin{align}
\label{eqn:ape1}|z_i(t)|&\leqslant  |\overline{z}_i|+ RT \quad &\forall\, t\in[0,T],\\
\label{eqn:ape2}|\dot{z}_i(t)|&\leqslant |\overline{z}_i|+ R(1+T)\quad &\forall\, t\in (0,T),
\end{align}
where $R=R(\Omega)>0$ is such that $\Omega\subset B_{R}(0)\subset \R^3$.
\end{lem}

\begin{proof}
Since the matrix $J$ is skew symmetric and has operator norm $1$, for any $i\in\{1,...,N\}$ we have
\begin{align*}
\done{}{t}|z_i|^2=2z_i\cdot \dot{z}_i=-2z_i\cdot Jx_i(\z)\leqslant 2R|z_i|
\end{align*}
on $[0,T]$. The maximal solution of the ODE-IVP
\begin{align*}
\begin{cases}
\done{y}{t}=2Ry^{1/2}\\
y(0)=|\oz_i|^2
\end{cases}
\end{align*}
on the interval $[0,T]$ is given by
\begin{align*}
y(t)=(|\oz_i|+Rt)^2.
\end{align*}
(Note that this is the unique solution unless $\oz_i=0$.) Applying a comparison lemma such as \cite[Theorem 13.2]{szarski1965differential} then establishes \eqref{eqn:ape1}. Hence,
\begin{align*}
|\dot{z}_i|=| J(z_i-x_i(\z))| \leqslant |z_i|+ |x_i(\z)| \leqslant |\oz_i|+ RT  +R
\end{align*}
on $[0,T]$, which establishes \eqref{eqn:ape2}.
\end{proof}

Using the regularity results and the a priori estimates established above, we now prove Proposition \ref{prop:ODEex}.

\begin{proof}[Proof of Proposition \ref{prop:ODEex}]
Combining Corollary \ref{cor:diffVx} with Lemma \ref{lem:diffOWM} we see that the optimal centroid map $\mathbf{x}$ is continuously differentiable on its domain $D$. The map $W$ is, therefore, locally Lipschitz on $D$. Now, let $\obz=(\oz_1,...,\oz_N)\in D$
satisfy \eqref{eqn:wellPrep}, let $\overline{M}>0$ be such that
\begin{align*}
\max_i\{|\overline{z}_i|\}<\overline{M}
\end{align*}
and let $T\in (0,\infty)$.  Then
\begin{align*}
\overline{B_{\frac{r}{2}}(\obz)}\subset D
\end{align*}
and, for all $\z\in B_{\frac{r}{2}}(\obz)$,
\begin{align*}
|W(\z)|&\leq |\z|+|\mathbf{x}(\z)|
\leq \frac{r}{2}+|\obz|+|\mathbf{x}(\z)|
\leq\frac{r}{2}+N^{\frac{1}{2}}\left(\overline{M}+R\right)\\
&\leq\frac{r}{2}+N^{\frac{1}{2}}\left(\overline{M}+R(1+T)\right)=:M.
\end{align*}
Let
\begin{align*}
T_0=\min\left\{T,\frac{r}{2M}\right\}.
\end{align*}
By the Picard-Lindel\"{o}f Existence Theorem (see, for example, \cite[Theorem 2.2]{teschl2012ordinary}) there exists a unique $C^1$-solution of \eqref{eqn:ODE-IVP} on the interval $[0,T_0]$.

Since the third row of the matrix $J$ is zero, by \eqref{eqn:wellPrep}
\begin{align*}
\min_{i\neq j}\vert(z_i(T_0)-z_j(T_0))\cdot e_3\vert\geqslant r.
\end{align*}
Hence
\begin{align*}
\overline{B_{\frac{r}{2}}(\z(T_0))}\subset D.
\end{align*}
By the a priori estimate \eqref{eqn:ape1} proved in Lemma \ref{lem:ape}, for all $\z\in B_{\frac{r}{2}}(\z(T_0))$,
\begin{align*}
|W(\z)|&\leq |\z|+|\mathbf{x}(\z)|
\leq \frac{r}{2}+|\z(T_0)|+|\mathbf{x}(\z)|
\leq\frac{r}{2}+N^{\frac{1}{2}}\left(\overline{M}+R(1+T_0)\right)\\
&\leq\frac{r}{2}+N^{\frac{1}{2}}\left(\overline{M}+R(1+T)\right)=M.
\end{align*}
Applying the Picard-Lindel\"{o}f Existence Theorem again we conclude that there exists a unique $C^1$-solution of \eqref{eqn:ODE-IVP} on the interval $[0,\min\{T,2T_0\}]$. Repeating this process, we obtain a unique $C^1$-solution $\mathbf{z}$ of \eqref{eqn:ODE-IVP} on the interval $[0,T]$. Since the function $W$ is $C^1$ on $\mathbf{z}([0,T])$, $\mathbf{z}$ is in fact $C^2$.
\end{proof}

To conclude this section we prove Theorem \ref{thm:1}. We begin by proving that any continuously differentiable discrete geostrophic solution is energy-conserving.

\begin{lem}\label{lem:discEnergyCons}
Any continuously differentiable discrete geostrophic solution is energy-conserving.
\end{lem}
\begin{proof}
Suppose that $\alpha\in C([0,T];\mathcal{P}_c(\R^3))$ is a continuously differentiable discrete geostrophic solution with initial measure 
\begin{align*}
\overline{\alpha}=\sum_{i=1}^Nm_i\delta_{\overline{z}_i}\in\mathcal{Q}^N(\R^3).
\end{align*}
By Lemma \ref{lem:odeChar}, $\alpha$ is then given by
\begin{align*}
\alpha_t=\sum_{i=1}^Nm_i\delta_{z_i(t)}\qquad \forall\, t\in [0,T],
\end{align*}
where, for each $i\in\{1,...,N\}$,
\begin{align}\label{eqn:energyODEIVP}
\begin{cases}
\dot{z}_i=J(z_i-x_i(\z)),\\
z_i(0)=\overline{z}_i.
\end{cases}
\end{align}
Moreover, letting $\widetilde{C}_i(t)=C_i(\z_i(t),\w_*(\z(t)))$ for $t\in[0,T]$, we have
\begin{align*}
\nabla\mathcal{B}[\alpha_t]=\sum_{i=1}^Nz_i(t)\mathds{1}_{\widetilde{C}_i(t)}.
\end{align*}
Then, by equation \eqref{eqn:energy},
\begin{align}\label{eqn:energyDiscrete}
E[\alpha_t]=\frac{1}{2}\sum_{i=1}^N\int_{\widetilde{C}_i(t)}|x-z_i(t)|^2\,\rd x-\frac{1}{2}\sum_{i=1}^Nm_i(z_i(t)\cdot e_3)^2-\frac{1}{2}\int_{\Omega}x_3^2\,\rd x.
\end{align}

We now show that the function $t\mapsto E[\alpha_t]$ is differentiable on the interval $(0,T)$ and that its derivative is zero, from which it follows that $\alpha$ is energy-conserving. Trivially, the third term of \eqref{eqn:energyDiscrete} is constant in time. By \eqref{eqn:energyODEIVP}, since the third row of the matrix $J$ is zero,
\begin{align*}
\done{}{t}(z_i(t)\cdot e_3)=0\quad \forall\,i\in\{1,...,N\},\, t\in(0,T).
\end{align*}
Therefore the time derivative of the second term of \eqref{eqn:energyDiscrete} is also zero on $(0,T)$. Now define the function $\zeta:\Omega\times[0,T]\to\R$ by
\begin{align*}
\zeta(x,t):=\frac{1}{2}\sum_{i=1}^N\mathds{1}_{\widetilde{C}_i(t)}(x)|x-z_i(t)|^2,
\end{align*}
so that the first term of \eqref{eqn:energyDiscrete} is precisely $\int_{\Omega}\zeta(x,t)\,\rd x$. Observe that for any fixed $s\in(0,T)$, the function $t\mapsto\zeta(x,t)$ is continuously differentiable at $s$ for $\mathcal{L}^3$-almost every $x$ in $\Omega$. Indeed, since the set $\partial \widetilde{C}_i(s)\cap \partial \widetilde{C}_j(s)$ is contained in a $2$-dimensional plane for each $i\neq j$ (see Section \ref{sect:SDOT}),
\begin{align*}
\mathcal{L}^3\left(\Omega\setminus \bigcup_{i=1}^N\mathrm{Int}\left({\widetilde{C}}_i(s)\right)\right)=\mathcal{L}^3\left(\bigcup_{i\neq j}\left(\partial \widetilde{C}_i(s)\cap \partial \widetilde{C}_j(s)\right)\cup \partial\Omega\right)=0.
\end{align*}
For $x$ in the interior of the cell $\widetilde{C}_i(t)$
\begin{align*}
\partial_t\zeta(x,s)&=-(x-z_i(s))\cdot \dot{z}_i(s)
=-(x-z_i(s))\cdot J(z_i(s)-x_i(\z(s))).
\end{align*}
By hypothesis, $\z(s)\in D$, so continuity at $s$ of the map $t\mapsto\partial_t\zeta(x,t)$ follows from Corollary \ref{cor:diffVx} and Lemma \ref{lem:diffOWM}. 
By Lemma \ref{lem:ape}, $\partial_t\zeta(\cdot,s)$ is uniformly bounded on $\Omega$. Hence, by the Mean Value Theorem and the Dominated Convergence Theorem, we may exchange the order of differentiation and integration and, using \eqref{eqn:energyODEIVP} and the skew-symmetry of the matrix $J$, we obtain
\begin{align*}
\done{}{t}\left(\int_{\Omega}\zeta(x,t)\,\rd x\right)\bigg\vert_{t=s}
&=\int_{\Omega}\partial_t\zeta(x,s)\,\rd x
=-\sum_{i=1}^N\int_{\widetilde{C}_i(s)}(x-z_i(s))\cdot \dot{z}_i(s)\,\rd x\\
&=-\sum_{i=1}^Nm_i\big(x_i(\z(s))-z_i(s)\big)\cdot J\big(z_i(s)-x_i(\z(s))\big)\,\rd x=0,
\end{align*}
which completes the proof.
\end{proof}

\begin{proof}[Proof of Theorem \ref{thm:1}]
Since $\overline{\alpha}$ is well-prepared in the sense of Definition \ref{defn:wellPrep}, by Proposition \ref{prop:ODEex} the corresponding ODE-IVP \eqref{eqn:ODE-IVP} has a unique $C^2$-solution. By Lemma \ref{lem:odeChar}, this gives rise to a unique twice continuously differentiable discrete geostrophic solution with initial measure $\overline{\alpha}$ via the formula \eqref{eqn:ansatzLem}. This solution is energy-conserving by Lemma \ref{lem:discEnergyCons}.
\end{proof}

\section{Proof of Theorem \ref{thm:2}}\label{sect:generalExistence}

We now construct a geostrophic solution with arbitrary initial measure $\overline{\alpha}\in \mathcal{P}_c(\R^3)$. We begin by proving the existence of a sequence $(\overline{\alpha}^N)_{N\in\N}$ of \emph{well-prepared} discrete probability measures (see Definition \ref{defn:wellPrep})  converging to $\overline{\alpha}$ with respect to the Wasserstein 2-distance (Lemma \ref{lem:construction}). The existence of a sequence of discrete geostrophic solutions $\alpha^N$ with initial measures $\overline{\alpha}^N$, respectively, is then guaranteed by Theorem \ref{thm:1}. To obtain a geostrophic solution with initial measure $\overline{\alpha}$, we then apply the Arzel\`a-Ascoli Compactness Theorem combined with the continuity of $\mathcal{B}$ (Theorem \ref{thm:quantStab}) and pass to the limit as $N\to\infty$.

\begin{lem}\label{lem:construction}
Let $\beta\in\mathcal{P}_c(\R^3)$. There exists a compact set $U\subset \R^3$, a sequence of discrete probability measures $\beta^N\in \mathcal{Q}^N(\R^3)$ given by
\begin{align*}
\beta^N=\sum_{i=1}^Nm^N_i\delta_{z^N_i}
\end{align*}
and a sequence of positive real numbers $r_N>0$ such that
\begin{align}
\label{eqn:betaSupp}&\mathrm{supp}(\beta^N)\subset U\quad \forall\, N\in\N,\\
\label{eqn:betaNLimit}&\underset{N\to\infty}\lim W_2(\beta^N,\beta)=0,
\end{align}
and
\begin{align}\label{eqn:LB}
|(z^N_i-z^N_j)\cdot e_3|\geqslant r_N \quad \forall\, i\neq j.
\end{align}
\end{lem}

\begin{proof}
Take a compact set $U \subset \mathbb{R}^3$ such that $\text{supp}(\beta)\subset U$ and a sequence of discrete probability measures 
\begin{align*}
\tilde{\beta}^N=\sum_{i=1}^Nm^N_i\delta_{\tilde{z}^N_i} \in\mathcal{Q}^N(\R^3),
\end{align*}
whose support is contained in $U$, such that
\begin{align}
\label{eq:QuantError}
\underset{N\to\infty}\lim W_2(\tilde{\beta}^N,\beta)=0.
\end{align}
Such a sequence exists by for example \cite[Lemma 10]{merigot2021optimal}.

For each $N \in \mathbb{N}$, let $z_i^N \in U$, $i \in \{1,\ldots,N\}$, satisfy the following: 
$(z_i^N-z_j^N)\cdot e_3 \ne 0$ for all $i,j\in\{1,...,N\}$, $i \ne j$; $z_i^N = \mathrm{argmin}_{z \in \{z_1^N,\ldots,z_N^N\}} |z - \tilde{z}^N_i|$
for all $i \in \{1,\ldots,N\}$; $|z_i^N - \tilde{z}^N_i|<1/N$ for all $i \in \{1,\ldots,N\}$. 
 Define 
\begin{align*}
\beta^N:=\sum_{i=1}^Nm_i^N\delta_{z_i^N}.
\end{align*}
Then
\begin{equation}
\label{eq:PertError}
W^2_2(\beta^N,\tilde{\beta}^N) = \sum_{i=1}^N m_i^N |z_i^N - \tilde{z}^N_i|^2 \le \frac{1}{N^2}.
\end{equation}
Combining \eqref{eq:QuantError} and \eqref{eq:PertError} completes the proof.
\end{proof}

Now fix $T\in (0,\infty)$ and $\overline{\alpha}\in\mathcal{P}_c(\R^3)$. By Lemma \ref{lem:construction} there exists a sequence of \emph{well-prepared} discrete probability measures $\overline{\alpha}^N\in\mathcal{Q}^N(\R^3)$, given by
\begin{align*}
\overline{\alpha}^N:=\sum_{i=1}^Nm^N_i\delta_{\overline{z}^N_i},
\end{align*}
and a positive constant $\overline{M}>0$ such that
\begin{align*}
\underset{N\to\infty}\lim W_2(\overline{\alpha}^N,\overline{\alpha})=0
\qquad
\text{and}
\qquad
\underset{N\in\N}{\bigcup}\text{supp}(\overline{\alpha}^N)\subset B_{\overline{M}}(0).
\end{align*}
By Theorem \ref{thm:1}, for each $N\in\N$ there exists a unique twice continuously differentiable discrete geostrophic solution $\alpha^N\in C([0,T];\mathcal{P}_c(\R^3))$ with initial measure $\alpha^N_0=\overline{\alpha}^N$, given by
\begin{align*}
\alpha^N_t=\sum^N_{i=1}m^N_i\delta_{z^N_i(t)} \quad \forall\, t\in [0,T],
\end{align*}
where $\z^N=(z^N_1,...,z^N_N):[0,T]\to\R^{3N}$ is a twice continuously differentiable solution of the ODE-IVP \eqref{eqn:ODE-IVP} with initial condition
\begin{align*}
\z^N(0)=(\oz_1^N,...,\oz_N^N).
\end{align*}

\begin{lem}[{\normalfont Compactness, c.f. \cite[Lemma 2.5]{loeper2006fully}, \cite[Theorem 3.2]{feldman2013lagrangian}}]\label{lem:compactness}
The sequence $(\alpha^N)_{N\in\N}$ has a uniformly convergent subsequence in $C\left([0,T];\mathcal{P}_{c}(\mathbb{R}^3)\right)$.
In particular, there exists a strictly increasing function $k:\mathbb{N}\to\mathbb{N}$ and a Lipschitz map $\alpha\in C^{0,1}\left([0,T];\mathcal{P}_{c}(\mathbb{R}^3)\right)$ such that
\begin{align*}
\underset{N\to\infty}\lim\sup_{t\in[0,T]}
W_2(\alpha_t^{k(N)},\alpha_t)=0.
\end{align*}
Moreover, there exists $R_1>0$ such that for all $t\in [0,T]$ and all $N\in\mathbb{N}$
\begin{align*}
\mathrm{supp}(\alpha^N_t),\, \mathrm{supp}(\alpha_t)\subset K:=\overline{B_{R_1}(0)}.
\end{align*}	
\end{lem}

\begin{proof}
The a priori estimates \eqref{eqn:ape1} on the paths $z^N_i$ mean that the metric space 
$(\mathcal{P}_{c}(\mathbb{R}^3),W_2)$ can be replaced by the compact metric space
$(\mathcal{P}(K),W_2)$,
where $K\subset\R^3$ is the closed ball of radius $R_1:=\overline{M}+R(\Omega)T$ centred at the origin. Since $K$ is compact, $W_1$ and $W_2$ are equivalent metrics on $\mathcal{P}(K)$ (see \cite[p.179]{santambrogio2015optimal}), so it is sufficient to prove that 
 $(\alpha^N)_{N\in\mathbb{N}}$
has a uniformly convergent subsequence in $C\left([0,T];(\mathcal{P}(K),W_1)\right)$ whose limit point is Lipschitz with respect to the $W_1$ metric. Moerover, since the arrival space $(\mathcal{P}(K),W_1)$ is a compact metric space, by the Ascoli-Arzel\'{a} Theorem (see for example \cite[p.10, Box 1.7]{santambrogio2015optimal}), such a sequence exists if and only if for every $\varepsilon > 0$, there exists $\delta(\varepsilon) >0$ such that, whenever $s,t \in [0,T]$ and $|t-s|<\delta(\varepsilon)$, $W_1(\alpha_t^{N},\alpha_s^{N})<\varepsilon$ for all $N$. 

For $\mu,\,\nu\in\mathcal{P}(K)$,
\begin{align}\label{eqn:W1char}
W_1(\mu,\nu) =
\sup \left\{
\int_K \phi \,   \rd(\mu-\nu) \
\Big\vert\ \phi: K \to \mathbb{R},\ \phi \text{ is 1-Lipschitz}
\right\}.
\end{align}
(See for example \cite[Theorem 1.14]{villani2003topics}.) Let $\phi: K \to \mathbb{R}$ be 1-Lipschitz and let $t,s\in[0,T]$. By the a priori estimate \eqref{eqn:ape2} we have
\begin{align*}
\int_K \phi \,   \rd(\alpha_t^{N}-\alpha_s^{N}) 
&=
\sum_{i=1}^Nm^N_i\Big(\phi\big(z^N_i(t)\big)-\phi\big(z^N_i(s)\big)\Big)
\\
& \le \sum_{i=1}^Nm^N_i \left| z^N_i(t) - z^N_i(s)\right|
\\
& \le \sum_{i=1}^Nm^N_i L |t-s|
\\
& = L |t-s|,
\end{align*}
where $L=\overline{M}+R(1+T)$. Using the characterisation of the Wasserstein 1-distance given by \eqref{eqn:W1char} we obtain
\begin{align}\label{eqn:alphaNLipEst}
W_1(\alpha^N_t,\alpha^N_s)\leq L|t-s|.
\end{align}
Choosing $\delta(\varepsilon)=\varepsilon/L$, the Ascoli-Arzel\'{a} Theorem guarantees the existence of a uniformly convergent subsequence $(\alpha^{k(N)})_{N\in\N}$. Denote by $\alpha$ its limit point. Combining the Lipschitz estimate \eqref{eqn:alphaNLipEst} and using the triangle inequality in $(\mathcal{P}(K),W_1)$, for all $N\in\N$ we have
\begin{align*}
W_1(\alpha_t,\alpha_s)\leq W_1(\alpha_t,\alpha_t^{k(N)})+W_1(\alpha_s,\alpha_s^{k(N)})+L|t-s|.
\end{align*}
Passing to the limit as $N\to\infty$, we see that $\alpha$ is Lipschitz with respect to $W_1$ as required.
\end{proof}

We conclude this section with the proof of Theorem \ref{thm:2}.

\begin{proof}[Proof of Theorem \ref{thm:2}]
Let $\alpha\in C^{0,1}([0,T];\mathcal{P}_c(\R^3))$ be the limit point of the sequence $(\alpha^{k(N)})_{N\in\N}$ obtained in Lemma \ref{lem:compactness}. First, we prove that $\alpha$ satisfies the transport equation \eqref{eqn:SGGdist}. For clarity, given $\beta\in C([0,T];\mathcal{P}(K))$ and $\varphi\in \mathscr{D}(\R^3\times\R)$ we define
\begin{align*}
\mathcal{F}[\beta,\varphi]:=
&\int_0^{T}\int_{K}\big(\partial_t \varphi(z,t)+Jz\cdot\nabla \varphi(z,t)\big)\,\rd\alpha_t(z)\,\rd t-\int_0^{T}\int_{\Omega}J x\cdot\nabla \varphi\big(\nabla \mathcal{B}[\alpha_t](x),t\big)\,\rd x\,\rd t\\
\nonumber&-\left(\int_{K}\varphi(z,T)\,\rd\alpha_{T}(z)-\int_{K}\varphi(z,0)\,\rd\alpha_{0}(z)\right).
\end{align*}
Since the space
\begin{align*}
\widetilde{\mathscr{D}}:=\{\varphi=\phi\psi\ \vert\ \phi\in C^{\infty}_c(\R),\ \psi\in C^{\infty}_c(\R^3)\}
\end{align*}
is a dense subspace of $\mathscr{D}(\R^3\times\R)$ (see \cite{friedlander1998introduction}), to show that $\alpha$ satisfies \eqref{eqn:SGGdist} it is enough to check that $\mathcal{F}[\alpha,\varphi]=0$ for any $\varphi\in \widetilde{\mathscr{D}}$. 
Moreover, for each $N\in\N$ we have
\begin{align*}
\mathcal{F}[\alpha^N,\varphi]&=0\qquad \forall\, \varphi\in \widetilde{\mathscr{D}}.
\end{align*}
By the triangle inequality, it is therefore sufficient to show that
\begin{align}\label{eqn:aim}
\underset{N\to\infty}\lim\big\vert \mathcal{F}[\alpha^{k(N)},\varphi]-\mathcal{F}[\alpha,\varphi]\big\vert=0\qquad \forall\, \varphi\in \widetilde{\mathscr{D}}.
\end{align}

Since $K$ is compact, $W_1$ and $W_2$ are equivalent metrics on $K$. The sequence $(\alpha^{k(N)})_{N\in\N}$ therefore converges to $\alpha$ uniformly in $C([0,T];(\mathcal{P}(K),W_1))$. By the characterisation of $W_1$ given by \eqref{eqn:W1char}, this implies that for any Lipschitz function $\eta:K\to\R$,
\begin{align}\label{eqn:unifConvLip}
\underset{N\to\infty}\lim\sup_{t\in[0,T]}\left\{\int_K\eta(z)\,\rd(\alpha^{k(N)}_t-\alpha_t)(z)\right\}=0.
\end{align}
For $\varphi=\phi\psi\in\widetilde{\mathscr{D}}$, where $\phi\in C^{\infty}_c(\R)$ and $\psi\in C^{\infty}_c(\R^3)$, we have
\begin{align*}
\int_0^{T}\int_{K}\partial_t \varphi(z,t)\,\rd(\alpha^{k(N)}_t-\alpha_t)(z)\,\rd t&= \int_0^{T} \phi^{\prime}(t)\left(\int_{K}\psi(z)\,\rd(\alpha^{k(N)}_t-\alpha_t)(z)\right)\,\rd t\\
&\leqslant T\sup_{t\in[0,T]}|\phi^{\prime}(t)|\sup_{t\in[0,T]}\bigg\vert\int_{K}\psi(z)\,\rd(\alpha^{k(N)}_t-\alpha_t)(z)\bigg\vert
\end{align*}
and
\begin{align*}
\int_0^{T}\int_{K}Jz\cdot\nabla \varphi(z,t)\,\rd(\alpha^{k(N)}_t-\alpha_t)(z)\,\rd t&=\int_0^{T} \phi(t)\left(\int_{K}Jz\cdot\nabla \psi(z)\,\rd(\alpha^{k(N)}_t-\alpha_t)(z)\right)\,\rd t\\
&\leqslant T\sup_{t\in[0,T]}|\phi(t)|\sup_{t\in[0,T]}\bigg\vert\int_{K}Jz\cdot\nabla \psi(z)\,\rd(\alpha^{k(N)}_t-\alpha_t)(z)\bigg\vert.
\end{align*}
Therefore, by \eqref{eqn:unifConvLip},
\begin{align}\label{eqn:conv1}
\underset{N\to\infty}\lim\int_0^{T}\int_{K}\big(\partial_t \varphi(z,t)+Jz\cdot\nabla \varphi(z,t)\big)\,\rd(\alpha^{k(N)}_t-\alpha_t)(z)\,\rd t=0.
\end{align}
Moreover, since uniform convergence implies pointwise convergence, we have
\begin{align}\label{eqn:conv2}
\underset{N\to\infty}\lim\left(\int_K\varphi(z ,T)\,\rd(\alpha_{T}-\alpha^{k(N)}_{T})(z)-\int_K\varphi(z ,0)\,\rd(\alpha_{0}-\alpha^{k(N)}_{0})(z)\right)=0.
\end{align}
Finally, let
\begin{align*}
L_{\varphi}=\underset{t\in[0,T]}\sup\underset{x\in\R^3}\sup |D^2\varphi(x,t)|
\end{align*}
and, again, let $R>0$ be such that $\Omega\subset B_R(0)$. By the Cauchy-Schwarz inequality in $L^2(\Omega;\R^3)$ and the Lipschitz continuity of $\mathcal{B}$ (Theorem \ref{thm:quantStab} and \cite[Theorem 3.1]{merigot2020quantitative}), for some constant $C>0$ depending only on $\Omega$ and $K$,
\begin{align*}
&\int_0^T\int_{\Omega}Jx\cdot\left(\nabla\varphi(\nabla \mathcal{B}[\alpha_t](x),t)-\nabla\varphi\big(\nabla \mathcal{B}[\alpha^{k(N)}_t](x),t\big)\right)\, \rd x\, \rd t\\
&\leqslant T R L_{\varphi}\sup_{t\in[0,T]}\|\nabla \mathcal{B}\left[\alpha_t\right]-\nabla \mathcal{B}[\alpha^{k(N)}_t]\|_{L^2(\Omega;\R^3)}\\
&\leqslant T R L_{\varphi}C\left(\sup_{t\in[0,T]}W_2(\alpha_t,\alpha^{k(N)}_t)\right)^{\frac{2}{15}}.
\end{align*}
Hence,
\begin{align}\label{eqn:conv3}
\underset{N\to\infty}\lim\int_0^{T}\int_{\Omega}Jx\cdot\left(\nabla\varphi(\nabla \mathcal{B}[\alpha_t](x),t)-\nabla\varphi\big(\nabla \mathcal{B}[\alpha^{k(N)}_t](x),t\big)\right)\, \rd x\, \rd t=0.
\end{align}
By combining \eqref{eqn:conv1}, \eqref{eqn:conv2} and \eqref{eqn:conv3}, we see that \eqref{eqn:aim} holds.

To complete the proof we note that each discrete solution $\alpha^{k(N)}$ is energy-conserving so, by continuity of the geostrophic energy functional $E$ on $\mathcal{P}(K)$ (see Remark \ref{rem:energy}), for any $t\in[0,T]$
\begin{align*}
E[\alpha_t]=\lim_{N\to\infty}E[\alpha^{k(N)}_t]=\lim_{N\to\infty}E[\overline{\alpha}^{k(N)}]=E[\overline{\alpha}],
\end{align*}
which means that $\alpha$ is also energy-conserving.
\end{proof}

\section{Explicit solutions}\label{sect:explicit}
Here we consider two special cases (Examples \ref{exa:singleMass} and \ref{exa:twoMasses}) where we can obtain explicit expressions for discrete geostrophic solutions. In Example \ref{exa:Leb} we show that $\mathcal{L}^3\mres\Omega$ is an equilibrium solution of \eqref{eqn:SGgeoActTrans} and that its optimal quantisers are equilibrium solutions of the corresponding ODE-IVP \eqref{eqn:ODE-IVP}.

\begin{example}[A single mass]\label{exa:singleMass}
{\normalfont First we consider the case of a single Dirac mass. This example has been discussed in both \cite[Section 5]{faria2009weak} and \cite[Section 2.2]{feldman2015semi} in the context of the $2$-dimensional semi-geostrophic equations on the physical domain $B_1(0)\subset \R^2$. In contrast with previous approaches, which use approximations of the Dirac mass by characteristic functions on balls, our solution follows immediately from the characterisation of discrete geostrophic solutions in terms of the ODE-IVP \eqref{eqn:ODE-IVP}. 

Let $\Omega\subset \R^3$ be open, bounded and convex, denote by $x_{\Omega}$ the centroid of $\Omega$, and let $\overline{z}\in\R^3$. By an argument analogous to the proof of Lemma \ref{lem:odeChar}, a map $\alpha\in C([0,T];\mathcal{P}_c(\R^3))$ given by
\begin{align*}
\alpha_t=\delta_{z(t)}
\end{align*}
is a $k$-times continuously differentiable discrete geostrophic solution with initial measure $\overline{\alpha}=\delta_{\overline{z}}$ if and only if
$z:[0,T]\to\R^3$ is a $k$-times continuously differentiable solution of the ODE
\begin{align}\label{eqn:1Dode}
\dot{z}=J(z-x_{\Omega})
\end{align}
satisfying the initial condition $z(0)=\overline{z}$. Such a map $z$ is unique and is given by
\begin{align*}
z(t)=\mathrm{e}^{tJ}\left(\overline{z}-x_{\Omega}\right)+x_{\Omega}\qquad \forall\, t\in [0,T].
\end{align*}
Moreover,
\begin{align*}
\nabla \mathcal{B}[\alpha_t]=z(t)\qquad \forall\, t\in [0,T].
\end{align*}
Let $z_3=z\cdot e_3$. By equation \eqref{eqn:energy}, the corresponding geostrophic energy satisfies
\begin{align*}
\done{}{t}\left(E[\alpha_t]\right)&=\done{}{t}\left(\frac{1}{2}\int_{\Omega}|x-z(t)|^2\, \rd x-\frac{1}{2}z^2_3(t)-\frac{1}{2}\int_{\Omega}x^2_3\, \rd x\right)\\
&=-\int_{\Omega}(x-z(t))\cdot \dot{z}(t)\, \rd x -z_3(t)\dot{z}_3(t)\\
&=(z(t)-x_{\Omega})\cdot J(z(t)-x_{\Omega})=0.
\end{align*}
Hence $\alpha_t$ is energy-conserving.}
\end{example}

\begin{example}[Two masses in a ball]\label{exa:twoMasses}
{\normalfont Let $\Omega\subset \R^3$ be the ball of volume $1$ centred at the origin, and let 
\begin{align*}
\overline{\alpha}=m\delta_{\oz_1}+(1-m)\delta_{\oz_2},
\end{align*}
where $\oz_{1}, \oz_{2}\in\R^3$ are distinct and $m\in (0,1/2]$. To construct a geostrophic solution with initial measure $\overline{\alpha}$ we make the ansatz
\begin{align*}
\alpha_t=m\delta_{z_1(t)}+(1-m)\delta_{z_2(t)},
\end{align*}
where each $z_i:[0,T]\to\R^3$ is continuously differentiable. This yields the ODE-IVP
\begin{align}\label{eqn:2massODEIVP}
\begin{cases}
\ \dot{z}_i=J(z_i-x_i(z_1,z_2)),\\
\ z_i(0)=\oz_i,
\end{cases}
\end{align}
for $i\in\{1,2\}$, where $x_i(z_1,z_2)$ denotes the centre of mass of the $i^{\text{th}}$ cell in the optimal Laguerre tessellation generated by $(z_1,z_2)$. 

Due to its simple shape, Laguerre tessellations of $\Omega$ which are generated by two seeds can be easily characterised. Indeed, for any given $(z_1,z_2)\in \R^{6}$ such that $z_1\neq z_2$, the boundary between the Laguerre cells $C_1$ and $C_2$ generated by $(z_1,z_2)$ is necessarily the intersection of the ball with a plane perpendicular to the vector $z_1-z_2$. Hence $C_1$ is the spherical cap of mass $m$ whose base has outward pointing normal vector $(z_2-z_1)/|z_2-z_1|$, and $C_2$ is its complement in $\Omega$. The centre of mass of the Laguerre cell $C_1$ must therefore lie some positive distance $r$, depending only on the mass $m$, along the axis defined by the vector $z_1-z_2$ about which $C_1$ is rotationally symmetric. Hence,
\begin{align}\label{eqn:x1eg}
x_1(z_1,z_2)=r(m)\frac{z_1-z_2}{|z_1-z_2|}.
\end{align}
Moreover, the centre of mass of the whole ball is the origin so
\begin{align}\label{eqn:x2eg}
mx_1+(1-m)x_2=0 \quad \implies \quad x_2=\frac{mr(m)}{m-1}\frac{z_1-z_2}{|z_1-z_2|}.
\end{align}

Using these observations, \eqref{eqn:2massODEIVP} becomes 
\begin{align}
&\dot{z}_1= J\left(z_1-r\frac{z_1-z_2}{|z_1-z_2|}\right),\label{eqn:z1eg}\\
&\dot{z}_2=J\left(z_2-\frac{mr}{m-1}\frac{z_1-z_2}{|z_1-z_2|}\right).\label{eqn:z2eg}
\end{align}
Subtracting \eqref{eqn:z2eg} from \eqref{eqn:z1eg} yields the following ODE-IVP for the difference $Z:=z_1-z_2$:
\begin{align}\label{eqn:Zode}
\begin{cases}
\dot{Z}=J\left(Z-q\frac{Z}{|Z|}\right),\vspace{6pt}\\
Z(0)=\oz_1-\oz_2=:\overline{Z},
\end{cases}
\end{align}
where
\begin{align*}
q=q(m)=\left(\frac{1}{1-m}\right)r(m)
\end{align*}
is a positive constant determined by $m$ alone. To solve \eqref{eqn:Zode}, we first note that due to the skew symmetry of $J$
\begin{align*}
\done{}{t}|Z|^2=2Z\cdot \dot{Z}=0\quad \implies\quad |Z(t)|=|\overline{Z}| \quad \forall\, \ t\in [0,T].
\end{align*}
Hence, \eqref{eqn:Zode} becomes
\begin{align*}
\begin{cases}
\dot{Z}(t)=\left(1-\frac{q}{|\overline{Z}|}\right)JZ(t),\\
Z(0)=\overline{Z}.
\end{cases}
\end{align*}
Its solution is the map $Z:[0,T]\to\R^3$ given by
\begin{align*}
Z(t)=e^{\omega t J}\overline{Z},
\end{align*}
where $\omega=1-q/|\overline{Z}|$. That is,
\begin{align*}
Z(t)=\left(\begin{array}{c c c}
\cos(\omega t) & -\sin(\omega t) & 0\\
\sin(\omega t)  & \cos(\omega t) & 0\\
0 & 0 & 1
\end{array}\right)\overline{Z}.
\end{align*}
Hence, equations \eqref{eqn:z1eg} and \eqref{eqn:z2eg} decouple and we obtain two linear inhomogeneous ODE-IVPs for $z_1$ and $z_2$. Noting that $\omega\neq 1$ since $q>0$, we use Duhamel's formula and find that
\begin{align*}
z_1(t)&=
\mathrm{e}^{tJ}\oz_1-\frac{r}{\omega-1}\left(e^{\omega tJ}-\mathrm{e}^{tJ}\right)\frac{\oz_1-\oz_2}{|\oz_1-\oz_2|},\\
z_2(t)&=
\mathrm{e}^{tJ}\oz_2+\frac{mr}{(1-m)(\omega-1)}\left(e^{\omega tJ}-\mathrm{e}^{tJ}\right)\frac{\oz_1-\oz_2}{|			\oz_1-\oz_2|}.
\end{align*}
To conclude, recalling the expressions for the centroids given by \eqref{eqn:x1eg} and \eqref{eqn:x2eg}, we note that $x_1$ and $x_2$ simply rotate anti-clockwise around the vertical coordinate axis with angular frequency $\omega$.}
\end{example}

\begin{example}[Equilibrium solutions]\label{exa:Leb}
{\normalfont The Lebesgue measure restricted to $\Omega$ is an equilibrium solution of \eqref{eqn:SGgeoActTrans}. Indeed, let $\alpha=\mathcal{L}^3\mres\Omega$. Then $\nabla\mathcal{B}[\alpha]=\mathrm{id}_{\Omega}$, which implies that $\nabla \mathcal{B}[\alpha]^*(x)=x\ \forall\, x\in\Omega$, which in turn implies that $\mathcal{W}[\alpha]=0$ on $\mathrm{supp}(\alpha)$. Let
\begin{align*}
\alpha^N=\sum_{i=1}^Nm_i\delta_{z_i}
\end{align*}
be an optimal quantiser of $\mathcal{L}^3\mres\Omega$. By, e.g., \cite[Proposition 3.1]{DuFaberGunzburger1999}, \cite[Corollary 4.3]{graf2007foundations}, the seeds $\z=(z_1,...,z_N)$ generate a \emph{centroidal Voronoi tessellation of $\Omega$}, i.e,
\begin{align*}
x_i(\z)=z_i\quad \forall\, i\in\{1,...,N\}.
\end{align*}
This means that $W(\z)=0$ so $\z$ is an equilibrium solution of the ODE-IVP \eqref{eqn:ODE-IVP}. The behaviour of $\mathcal{L}^3\mres\Omega$ under the dynamics of equation \eqref{eqn:SGgeoActTrans} therefore agrees exactly with that of its optimal discrete approximants.
}
\end{example}

\section{Numerical simulations}\label{sect:num}
One advantage of the constructive existence proof given in this paper is that it naturally leads to a numerical method (a meshfree method) and moreover it tells us something about the convergence of this method. To be precise, given an initial measure $\overline{\alpha} \in \mathcal{P}_c(\mathbb{R}^3)$, we can construct a numerical approximation of a solution of the semi-geostrophic equations \eqref{eqn:SGGdist} as follows:
\begin{enumerate}[leftmargin=*]
\item Approximate $\overline{\alpha}$ by a discrete measure $\overline{\alpha}^N=\sum_{i=1}^N m_i \delta_{\overline{z}_i}$. This leads to the semi-discrete numerical scheme (continuous in time, discrete in space) used in the proof of Theorem \ref{thm:2}, where an exact solution  $\alpha^N=\sum_{i=1}^N m_i \delta_{z_i(t)}$ of \eqref{eqn:SGGdist} with initial condition $\overline{\alpha}^N$ is constructed by solving the system of ODEs $\dot{\mathbf{z}}=W(\mathbf{z})$ given in equation \eqref{eqn:ODE-IVP}.
    To turn this into a bona fide numerical method we require a further discretisation:
\item Use a time-stepping scheme to approximately solve the ODE $\dot{\mathbf{z}}=W(\mathbf{z})$. Every evaluation of the vector field $W$ requires a further numerical approximation; to evaluate $W(\mathbf{z}(t))$ we must solve the semi-discrete transport problem $W_2(\mathcal{L}^3 \mres \Omega, \alpha_t)$
    in order to compute the centroids $\mathbf{x}(\mathbf{z}(t))$.
\end{enumerate}

We demonstrate the viability of this numerical method by giving an example in two dimensions, implemented in MATLAB. Due to a lack of space we postpone the three-dimensional implementation and a more thorough numerical study to a further paper. Before giving the example we briefly discuss some implementation issues and convergence.

\paragraph{ODE solver.} We used the explicit Runge-Kutta scheme RK4 \cite[Example 5.13]{LeVeque2007} to solve the ODE $\dot{\mathbf{z}}=W(\mathbf{z})$. For larger simulations it may be better to use a linear multistep method \cite[Section 5.9]{LeVeque2007} since each evaluation of the vector field $W$ is expensive (linear multistep methods only require one new vector field evaluation per time step, whereas RK4 requires four per time step).

\paragraph{Semi-discrete transport solver.} The semi-discrete transport problem $W_2(\mathcal{L}^3 \mres \Omega, \alpha_t)$ can be solved by maximising the concave function $g$ (see equation \eqref{eq:g}) as described in Section \ref{sect:SDOT}. We did this in MATLAB using a quasi-Newton method (the MATLAB function \emph{fminunc}). Every evaluation of $g$ and its gradient requires a Laguerre diagram to be computed, which we did using the MATLAB function \emph{power\_bounded} \cite{PowerBounded}. This simple method, which is described in more detail in \cite[Algorithm 1 and Section 4]{BourneKokRoperSpanjer2020}, was sufficient for our proof of concept simulations here. In general, however, to maximise $g$ it would be much faster to use the damped Newton method from \cite{kitagawa2016convergence}, especially for large 3D simulations.

\paragraph{Convergence.} By Theorem \ref{thm:2} the sequence $(\alpha^N)_{N \in \mathbb{N}}$ generated by the semi-discrete scheme has a subsequence that converges (uniformly with respect to the Wasserstein metric) to a solution of the semi-geostrophic equations \eqref{eqn:SGGdist} with initial condition $\overline{\alpha}$. In particular, if equation \eqref{eqn:SGGdist} has a \emph{unique} weak solution with initial measure $\overline{\alpha}$, then the whole sequence of approximations $\alpha^N$ converges to the true solution. Local-in-time uniqueness of H\" older continuous periodic solutions of \eqref{eqn:SGGdist} was proved in \cite{loeper2006fully}, but is not known in general. By the conservation of the transport cost $W_2(\mathcal{L}^3 \mres \Omega, \alpha_t)$ (see Remark \ref{rem:ConsTransCost}) it is easy to see that the whole sequence $(\alpha^N)_{N\in\N}$ also converges in the very special case where the initial measure $\overline{\alpha}$ is the Lebesgue measure on $\Omega$. (Note that the Lebesgue measure is an equilibrium solution of \eqref{eqn:SGGdist} as discussed in Example \ref{exa:Leb}.) In general, proving convergence of the whole sequence $(\alpha^N)_{N\in\N}$ is beyond the scope of this paper, as is proving convergence of the fully discrete scheme. We will study these in a future paper.

\begin{figure}
\begin{subfigure}[b]{\textwidth}
    \centering
    \includegraphics[trim={0.5cm 0cm 0.75cm 0cm},clip,width=0.32\textwidth]{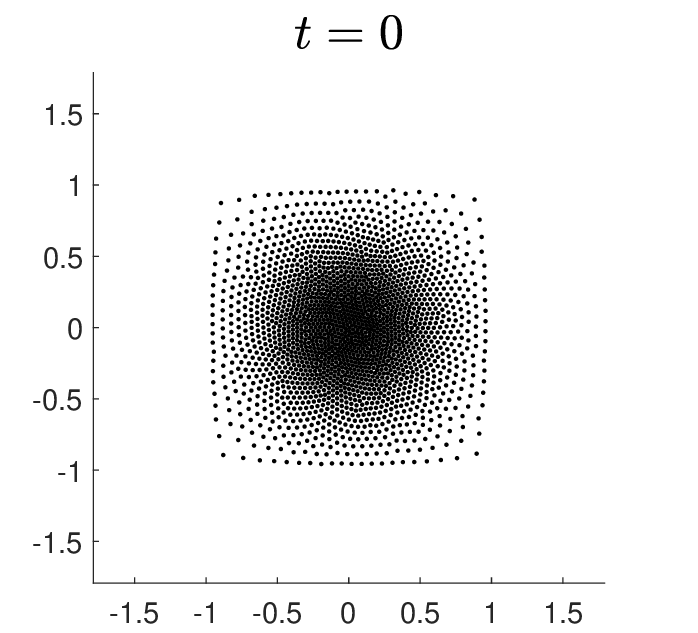}
\includegraphics[trim={0.5cm 0cm 0.75cm 0cm},clip,width=0.32\textwidth]{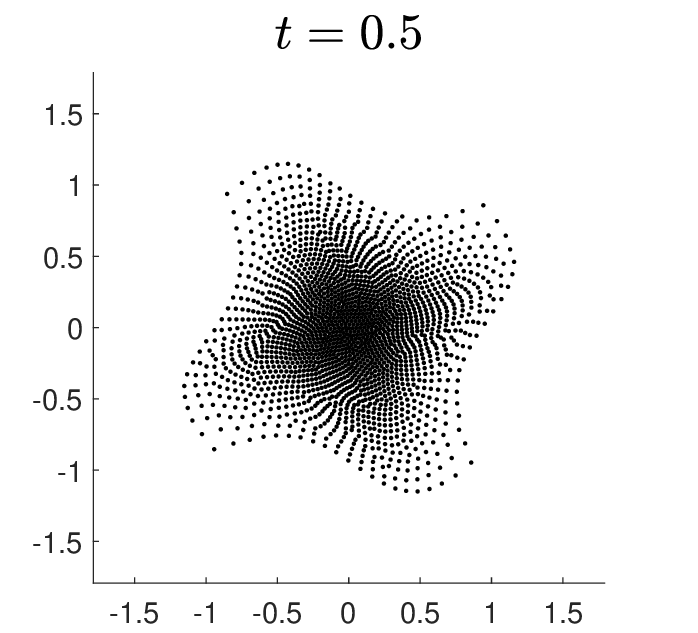}
  \includegraphics[trim={0.5cm 0cm 0.75cm 0cm},clip,width=0.32\textwidth]{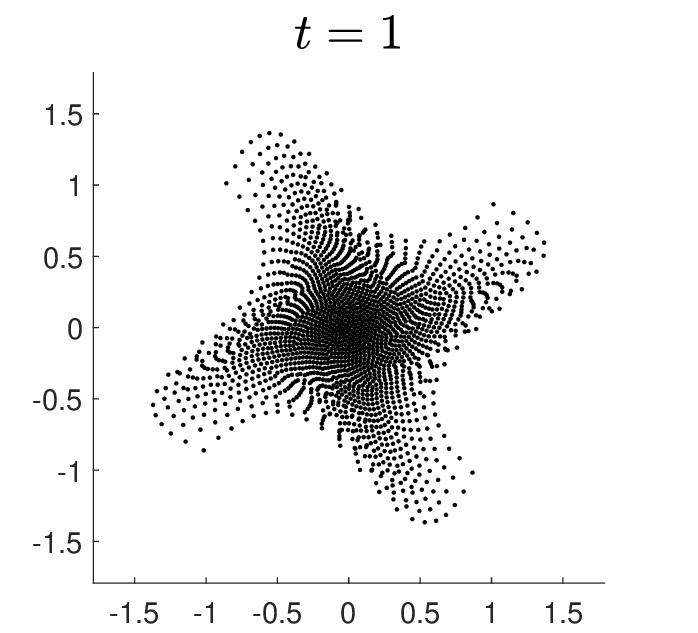}

\includegraphics[trim={0.5cm 0cm 0.75cm 0cm},clip,width=0.32\textwidth]{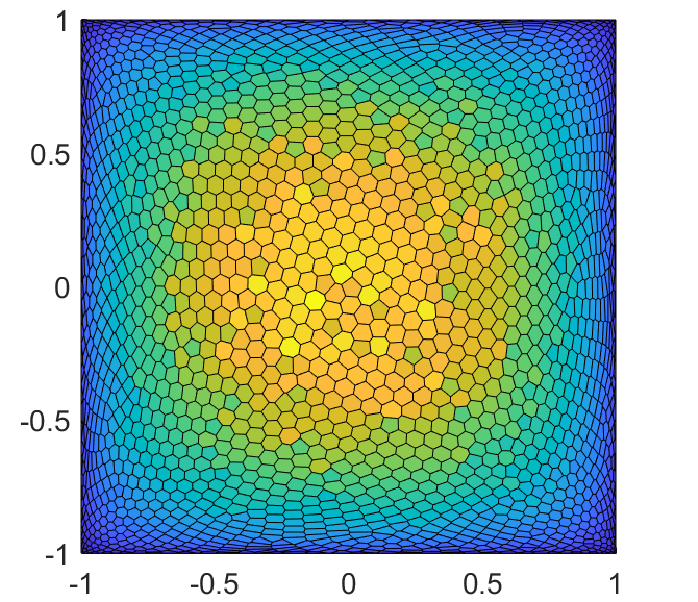}
\includegraphics[trim={0.5cm 0cm 0.75cm 0cm},clip,width=0.32\textwidth]{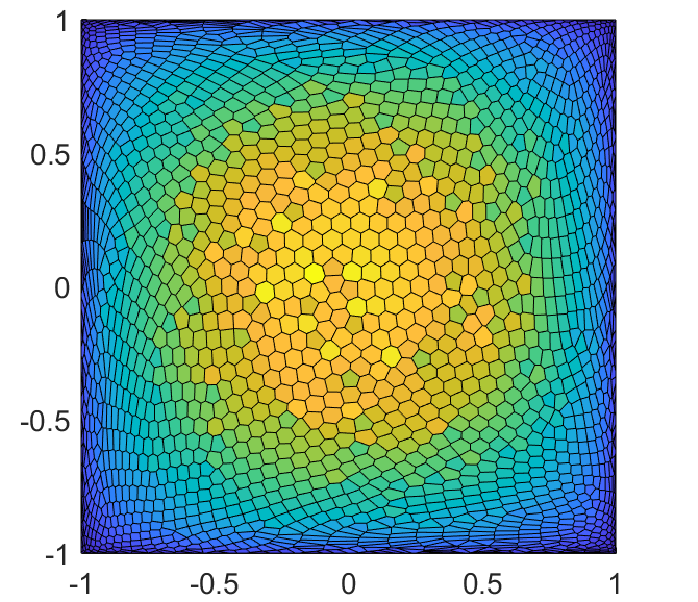}
  \includegraphics[trim={0.5cm 0cm 0.75cm 0cm},clip,width=0.32\textwidth]{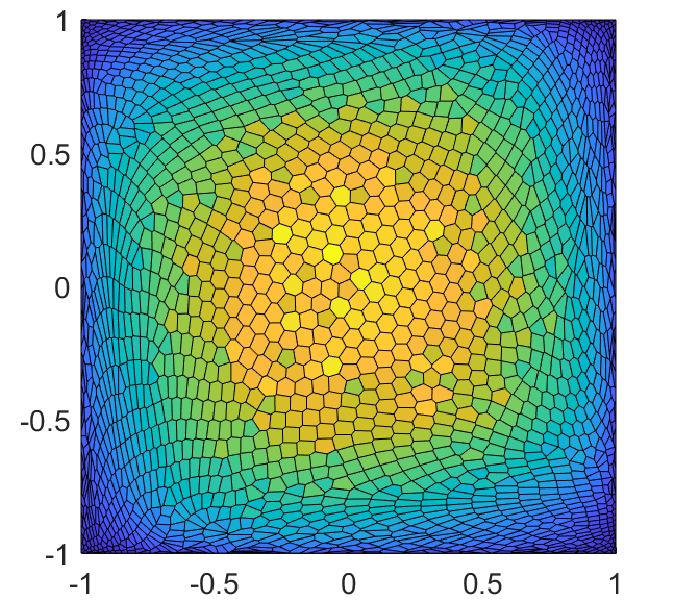}
  \end{subfigure}
  \vspace{0.03cm}

\begin{subfigure}[b]{\textwidth}
    \centering
  \includegraphics[trim={0.5cm 0cm 0.75cm 0cm},clip,width=0.32\textwidth]{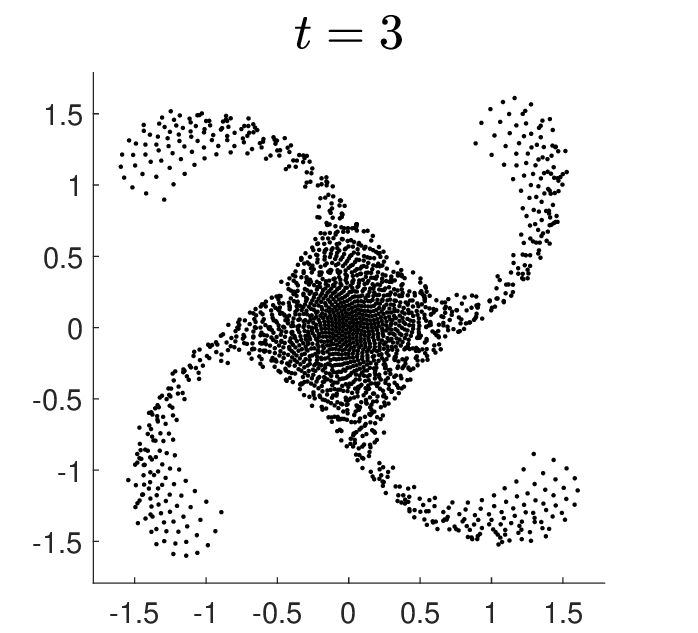}
\includegraphics[trim={0.5cm 0cm 0.75cm 0cm},clip,width=0.32\textwidth]{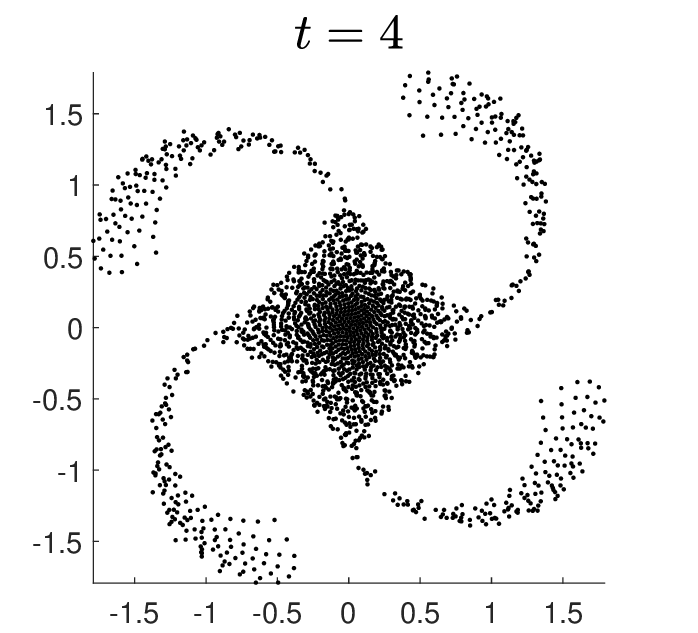}
\includegraphics[trim={0.5cm 0cm 0.75cm 0cm},clip,width=0.32\textwidth]{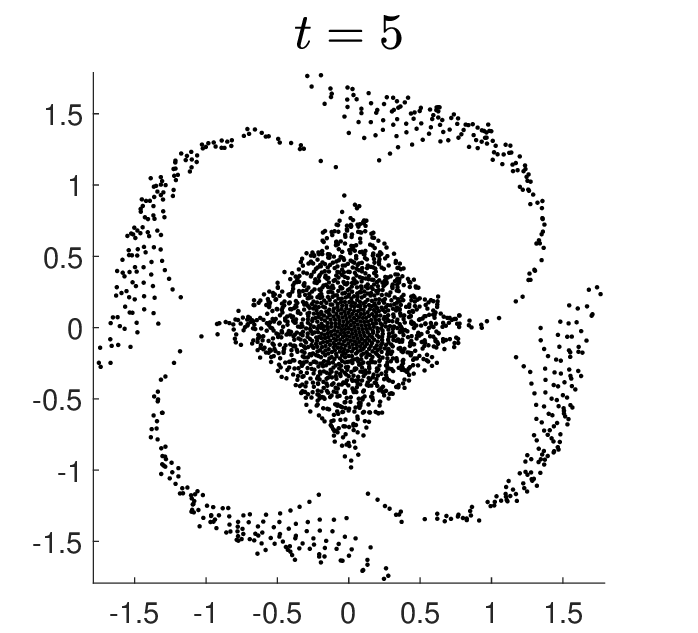}

  \includegraphics[trim={0.5cm 0cm 0.75cm 0cm},clip,width=0.32\textwidth]{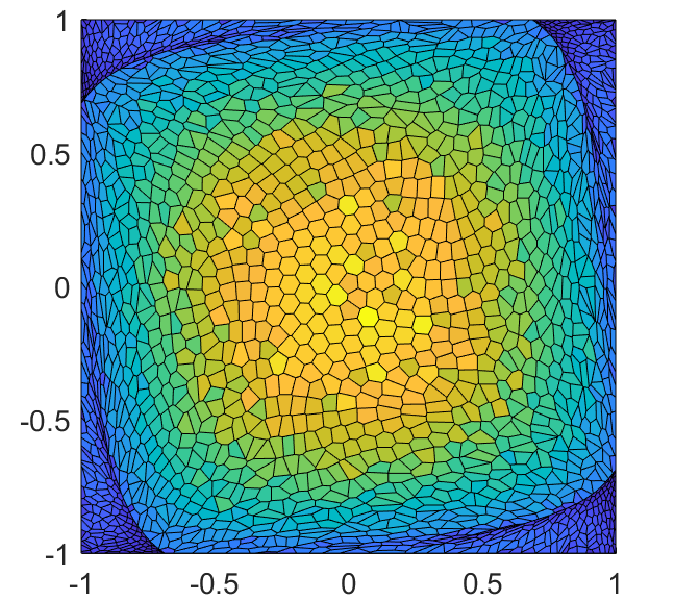}
\includegraphics[trim={0.5cm 0cm 0.75cm 0cm},clip,width=0.32\textwidth]{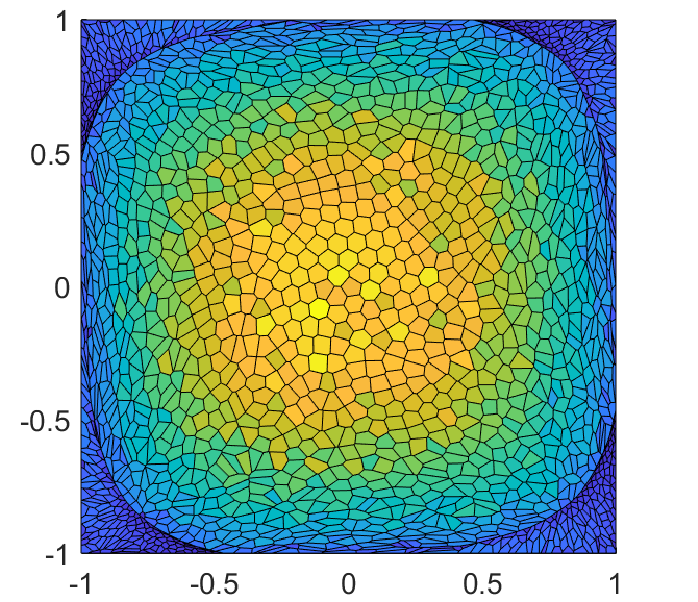}
\includegraphics[trim={0.5cm 0cm 0.75cm 0cm},clip,width=0.32\textwidth]{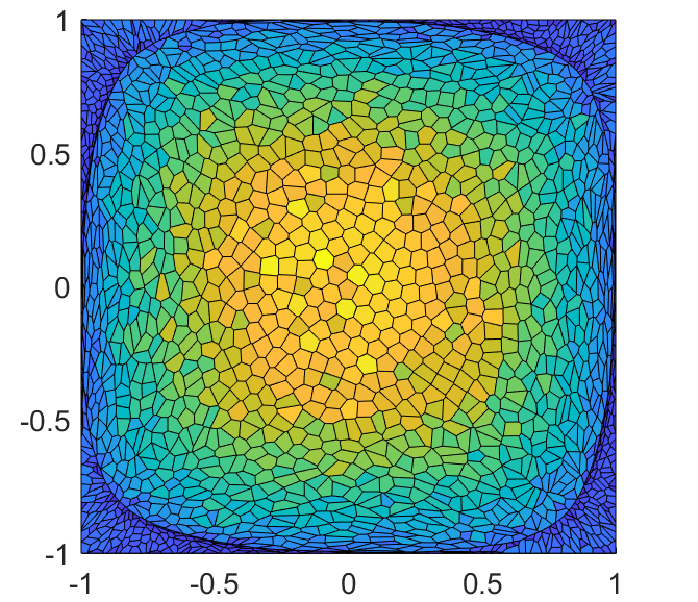}
  \end{subfigure}
  \caption{\label{Fig:Gaussian} Results of Example \ref{Example:Gaussian}. The seeds $z_i$ (black dots) and the corresponding Laguerre cells (coloured polygons) are illustrated at time steps $t=0,0.5,1,3,4,5$ for a Gaussian initial measure. The Laguerre cells are coloured according to their area (small cells in blue, large cells in yellow).}
  \end{figure}

\begin{example}[Gaussian initial condition]
\label{Example:Gaussian}
{\normalfont
Let $\Omega = [-1,1] \times [-1,1]$ and let the initial measure $\overline{\alpha}$ be absolutely continuous with respect to the Lebesgue measure with density $C \exp(-|x|^2) \mathds{1}_{\Omega}$, where $C>0$ is the normalisation constant satisfying
\[
\overline{\alpha}(\mathbb{R}^2) = \mathcal{L}^2(\Omega) \quad \Longleftrightarrow \quad C = 4 \left( \int_\Omega \exp(-|x|^2) \,  \rd x \right)^{-1}.
\]
Note that we have dropped the previous normalisation convention that $\overline{\alpha}(\mathbb{R}^2) = \mathcal{L}^2(\Omega)=1$, which is not necessary for the analysis; we simply require that $\overline{\alpha}(\mathbb{R}^2) = \mathcal{L}^2(\Omega)$. We approximated $\overline{\alpha}$ by a discrete measure of the form $\overline{\alpha}^N = \sum_{i=1}^N m_i \delta_{\overline{z}_i}$ with $N=2000$ seeds using $1000$ iterations of Lloyd's algorithm \cite{DuFaberGunzburger1999}.
We approximated the solution of the ODE-IVP \eqref{eqn:ODE-IVP} on the time interval $[0,T]$ with $T=5$ using the Runge-Kutta method RK4 with uniform time step size $h=0.01$. For the numerical maximisation of $g$ we used the following stopping condition: for all $i \in \{1,\ldots,N\}$,
\[
\left| \frac{\partial g}{\partial w_i} \right| < 10^{-2} \, \varepsilon \min_{j} m_j \quad \Longleftrightarrow \quad  |m_i-|C_i|| < 10^{-2} \, \varepsilon \min_{j} m_j
\]
with $\varepsilon=0.1$. This ensures that the areas of the Laguerre cells $C_i$ are correct to within $0.1 \%$.

The results are shown in Figure \ref{Fig:Gaussian} at time steps $t=0,0.5,1,3,4,5$. The black dots in the first and third rows are the (approximate) seed locations $z_i$. The polygons in the second and fourth rows are the (approximate) Laguerre cells $C_i$. The cells are coloured according to their target areas $m_i$, where blue corresponds to small cells and yellow corresponds to large cells. Note that all the seeds start off in $\Omega$ but they are not confined there. 

As an accuracy check, we repeated these simulations with a smaller time step size of $h=0.005$ and a finer optimal transport tolerance of $\varepsilon=0.05$. We found that the $x$- and $y$-coordinates of the seeds $z_i$ at the final time step $T=5$ agreed with our previous results to within $10^{-3}$. Recall from Remark \ref{rem:ConsTransCost} that the exact dynamics \eqref{eqn:ODE-IVP} preserves the transport cost:
\[
\frac{d}{dt} W_2(\mathcal{L}^2 \mres \Omega, \alpha_t) = 0.
\]
In our simulations (with $h=0.01$, $\varepsilon=0.1$) the transport cost was conserved to within $7.5 \times 10^{-7}$:
\[
\max_{n} \left| W_2^2(\mathcal{L}^2 \mres \Omega, \alpha_{t_n}) - \frac{1}{501}\sum_{m=0}^{500} W_2^2(\mathcal{L}^2 \mres \Omega, \alpha_{t_m}) \right| < 7.5 \times 10^{-7}
\]
where $t_n = nh$, $n \in \{0,1,\ldots,500\}$.}
\end{example}

\section*{Acknowledgements}
We thank both Mike Cullen and Colin Cotter for stimulating discussions related to this work. DPB would like to thank the UK Engineering and Physical Sciences Research Council (EPSRC) for financial support via the grant EP/R013527/2 Designer Microstructure via Optimal Transport Theory. CPE is supported by The Maxwell Institute Graduate School in Analysis and its Applications, a Centre for Doctoral Training funded by the EPSRC (grant EP/L016508/01), the Scottish Funding Council, Heriot-Watt University and the University of Edinburgh. BP and MW gratefully acknowledge the support of the EPSRC via the grant EP/P011543/1 Analysis of models for large-scale geophysical flows.

\newpage

\bibliographystyle{spmpsci}

\bibliography{SGbib}

\end{document}